\documentclass[12pt]{article}
\newcommand{\thetitle}{Coordination of multi-agent systems via asynchronous cloud communication}

\usepackage[T1]{fontenc}
\usepackage{lmodern}
\usepackage[utf8]{inputenc}

\usepackage{microtype}

\usepackage{ellipsis}

\usepackage{amsmath,amssymb}
\usepackage{amsthm}
\usepackage{mathtools}
\usepackage{graphicx}
\usepackage{pgfplots}
\usepackage{algorithm,algorithmicx,algpseudocode}
\usepackage{psfrag}

\usepackage{makecell}
	
\usepackage{color}
\usepackage{xspace}
\usepackage{subfigure}
\usepackage{accents}

\usepackage{authblk}

    \newcommand{\algoFull}{\texttt{\thetitle}\xspace}

\newcommand{\tlast}{t^{\textrm{last}}}
\newcommand{\tnext}{t^{\textrm{next}}}
\newcommand{\texp}{t^{\textrm{expire}}}
\newcommand{\tdwell}{T^{\textrm{dwell}}}
\newcommand{\tideal}{t^{\textrm{ideal}}}
\newcommand{\tmax}{t_\textrm{max}}

\newcommand{\plast}{p^\textrm{last}}

\newtheorem{theorem}{Theorem}[section]

\newtheorem{remark}[theorem]{Remark}

\newtheorem{proposition}{Proposition}

\newtheorem{definition}{Definition}

\newtheorem{problem}{Problem}

\newcommand{\real}{{\mathbb{R}}}

\newcommand{\integernonnegative}{\mathbb{Z}_{\ge 0}}

\newcommand{\GG}{{\mathcal{G}}}

\newcommand{\NN}{{\mathcal{N}}}

\renewcommand{\epsilon}{\varepsilon}

\newcommand{\argmax}{\operatorname{argmax}}

\newcommand{\until}[1]{\{1,\dots, #1\}}

\newcommand{\TwoNorm}[1]{\|#1\|_2}
\renewcommand{\TwoNorm}[1]{\|#1\|}

\newcommand{\commgraph}{\GG_\text{comm}}

\renewcommand{\commgraph}{\GG}

\newcommand{\timelast}{t^\text{last}}
\newcommand{\timenext}{t^\text{next}}

\newcommand{\oprocendsymbol}{\hbox{$\bullet$}}
\newcommand{\oprocend}{\relax\ifmmode\else\unskip\hfill\fi\oprocendsymbol}

\title{Coordination of multi-agent systems via asynchronous cloud communication}

\author[1]{Sean L. Bowman}
\author[2]{Cameron Nowzari}
\author[3]{George J. Pappas}

\affil[1]{Computer and Information Science Department, University of Pennsylvania, Philadelphia, PA, USA}
\affil[2]{Electrical and Computer Engineering Department, George Mason University, Fairfax, VA, USA}
\affil[3]{Electrical and Systems Engineering Department, University of Pennsylvania, Philadelphia, PA, USA}

\begin{document}
  
\maketitle

\begin{abstract}
In this work we study a multi-agent coordination problem
in which agents are only able to communicate with each other intermittently 
through a cloud server. To reduce the amount of
required communication, we develop a self-triggered algorithm
that allows agents to communicate with the cloud only when necessary
rather than at some fixed period.
Unlike the vast majority of similar works that propose
distributed event- and/or self-triggered control laws, this work
doesn't assume agents can be ``listening'' continuously. In other
words, when an event is triggered by one agent, neighboring agents
will not be aware of this until the next time they establish communication
with the cloud themselves. Using a notion of ``promises'' about future control inputs, agents are able to keep track of higher quality estimates about their neighbors allowing them to stay disconnected from the cloud for longer periods of time while still guaranteeing a positive contribution to the global task. We prove that our self-triggered coordination algorithm
guarantees that the system asymptotically reaches the set of desired states.
Simulations illustrate our results.
\end{abstract}

\section{Introduction}

This paper considers a multi-agent coordination problem where agents
can only communicate with one another indirectly through the use of a central base
station or ``cloud.''
Small connected household devices that require communication and coordination with each other are become increasingly prevalent (the ``Internet of Things'').
To reduce both power consumption and bandwidth requirements for these small, low-power devices, it is ideal that they communicate as infrequently as possible with the cloud server. For instance, one can imagine a number of devices trying to coordinate through asynchronous communication with a dedicated cloud (e.g., email) server. In this setting, a device can only receive and send messages while connected to the server; however, being connected to the server at all times is a waste of energy and wireless resources. %
In this paper we present a method to facilitate the coordination of a number of agents through a cloud server that guarantees the completion of a global task while reducing the number of communications required and without the need for a device to continuously be in communication with the cloud server.

Specifically, we consider the more concrete related problem of coordinating a number of submarines that only can communicate with a base station while at the surface of the water.
While a majority of related works allow for an agent to push information to its neighbors at any desired time, communicating with the outside
world when underwater is extremely expensive, if not 
impossible~\cite{NAC-BMF-OK-ACM-CP-RP-DS:13,EF-NEL-PB-DAP-RB-DMF:06}, and so a submarine must perform all communication while surfaced. 

Each time a submarine surfaces, it must determine the next time to surface as well as the control law to use while underwater in order to adequately achieve some desired global task based only on information available on the server at that moment.
In this paper we are interested in designing a self-triggered coordination algorithm in which agents can autonomously schedule the next time to communicate
with the cloud based on currently available information. 
While we motivate our problem via an underwater
coordination problem in which communication while submerged is 
impossible, it is directly applicable to any scenario where
wireless-capable agents cannot be listening to communication 
channels continuously.

\emph{Literature review:}
In the context of the multi-agent coordination problem in general, the literature is extensive~\cite{olfati-saber07,ren-book,MM-ME:10}.
In our specific problem of multi-agent consensus, Olfati-Saber and Murray~\cite{olfati-saber03} introduce a continuous-time law that guarantees consensus convergence on undirected as well as weight-balanced digraphs.
However, the majority of these works assume agents can continuously, or at least periodically, obtain information about their neighbors.
Instead, when communication is expensive as in our case, we wish to minimize the number of times communication is necessary. 

A useful tool for determining discrete communication times in this manner is event-triggered control, where an algorithm is designed
to tune controller executions to the state evolution of a given system, see
e.g.,~\cite{KJA-BMB:02,WPMHH-KHJ-PT:12}. 
In particular, event-triggered control has been successfully applied to multi-agent systems with the goal of limiting computation and decision making to reduce overall communication, sensing, and/or actuation effort of the agents. 
In~\cite{MMJ-PT:10}, the authors formulate a threshold on system error to determine when control signals need to be updated. In~\cite{XW-MDL:11}, the authors expand on this and determine a distributed threshold for a wireless control network, further taking into account network errors such as communication delays and packet drops. Event-triggered ideas have also been applied to the acquisition of information rather than control.
Several approaches~\cite{GX-HL-LW-YJ:09,WPMHH-MCFD:13,XM-TC:13} utilize periodically sampled data to reevaluate the controller trigger. 
Zhong and Cassandras~\cite{MZ-CGC:10} additionally drop the need for periodic sampling, creating a distributed trigger to decide when to share data based only on local information. 

Event-triggered approaches generally require the persistent monitoring of some
triggering function as new information is being obtained. Unfortunately, this is not directly
applicable to our setup because the submarines only get new information when they are at the surface of the water.
Instead, self-triggered control~\cite{AA-PT:10,XW-MDL:09,CN-JC:11-auto} removes the need to continuously monitor the triggering function, instead requiring each agent to compute its next trigger time based solely on the information available at the previously triggered sample time. 

The first to apply these ideas to consensus, %
Dimarogonas et al.~\cite{DVD-EF-KHJ:12}, remove the need for continuous control by introducing an event-triggered rule to determine when an agent should update its control signal, however still requiring continuous information about their neighbors.
In~\cite{GSS-DVD-KHJ:13}, the authors further remove the need for continuous neighbor state information, creating a time-dependent triggering function to determine when to broadcast information.
The authors in~\cite{EG-YC-HY-PA-DC:13} similarly broadcast based on a state-dependent triggering function.
Recently, these ideas have been extended from undirected graphs to arbitrary directed ones~\cite{XM-TC:13,CN-JC:14-acc,XM-LX-YCS-CN-GJP:15}.

A major drawback of all aforementioned works is that they require all agents to be ``listening,'' or available to receive information, at all times.
Specifically, when any agent decides to broadcast information to its neighbors, it is assumed that all neighboring agents in the communication graph are able to instantaneously receive that information. Instead, we are interested in a situation where when an agent is disconnected from the cloud, it is incapable of communicating with other agents. 

In~\cite{PVT-DVD-KHJ-JS:10}, the authors study a very similar problem to the
one we consider here but develop an event-triggered solution in which all Autonomous Underwater Vehicles (AUVs) must surface together at the same time. Instead, we are interested in a strategy in which AUVs can autonomously surface asynchronously while still guaranteeing a desired stability property. This problem has very recently been looked at in~\cite{AD-DL-DVD-KHJ:15,CN-GJP:16,adaldo_cdc16} where the authors utilize event- and self-triggered coordination strategies to determine when the AUVs should resurface. In~\cite{AD-DL-DVD-KHJ:15}, a time-dependent triggering rule $\beta(\sigma_0, \sigma_1, \lambda_0, t)$ is developed that ensures practical convergence (in the presence of noise) of the whole system to the desired configuration. 
In~\cite{adaldo_cdc16} the authors present a similarly time-dependent triggering rule that allows agents to track a reference trajectory in the presence of noise. Instead, the authors in~\cite{CN-GJP:16} develop a state-dependent triggering rule with no explicit dependence on time; however, the self-triggered algorithm developed there is not guaranteed to avoid Zeno behaviors which makes it an incomplete solution to the problem. In this work we incorporate ideas of promises from team-triggered control~\cite{CN-JC:15-crc,CN-JC:16-tac} to develop a state-dependent triggering rule that guarantees asymptotic convergence to consensus while ensuring that Zeno behavior is avoided. 

\emph{Statement of contributions:} 
Our main contribution is the development of a novel distributed team-triggered algorithm
that combines ideas from self-triggered control with a notion of ``promises.'' These promises allow agents to make better decisions since they have higher quality information about their neighbors in general. Our algorithm incorporates these promises into the state-dependent trigger
to determine when they should communicate with the cloud. 
In contrast to~\cite{AD-DL-DVD-KHJ:15,CN-GJP:16}, our algorithm uses a state-dependent triggering rule with no explicit dependence on time, no global parameters, and no possibility of Zeno behavior. The main drawback of the time-dependent triggering rule~$\beta(\sigma_0, \sigma_1, \lambda_0, t)$ is that the choice of the constants $\sigma_0, \sigma_1, \lambda_0$ greatly affect the performance (number of events and convergence speed) of the system and there is no good way to choose these a priori; i.e., depending on the initial condition, different values of $\sigma_0, \sigma_1, \lambda_0$ will perform better. Instead, the state-dependent triggering rule developed here is more naturally coupled with the current state of the system. 
In general, distributed event- and self-triggered algorithms are designed so that agents
are \emph{never} contributing negatively to the global task, generally defined by the
evolution of a Lyapunov function~$V$. 
Instead,
our algorithm does not rely on this guarantee. More specifically, 
we actually allow an agent to be contributing negatively to the 
global task temporarily as long as it is accounted for by its net contribution over time.
Our algorithm guarantees the system converges asymptotically to consensus while ensuring that Zeno executions cannot occur. 
Finally, we illustrate our results through simulations.
\section{Problem Statement}

We consider system of $N$ submarine agents with single-integrator dynamics
\begin{align}\label{eq:dynamics}
  \dot{x}_i(t) = u_i(t),
\end{align}
for all $i \in \until{N}$, where we are interested
in reaching a consensus configuration, i.e. where 
$\TwoNorm{ x_i(t) - x_j(t) } \rightarrow 0$
as $t \rightarrow \infty$ for all $i, j \in \until{N}$. 
For simplicity, we consider scalar states $x_i \in \mathbb{R}$, 
but these ideas are extendable to arbitrary dimensions. 

Given a connected communication graph $\mathcal{G}$, it is well known~\cite{olfati-saber03} that the distributed continuous control law 
\begin{align}
u_i(t) &= -\sum_{j\in\mathcal{N}_i} \left( x_i(t) - x_j(t) \right) \label{eq:continuouslaw}
\end{align}
drives each agent of the system to asymptotically converge to the
average of the agents' initial conditions. In compact form, this can
be expressed by
\begin{align*}
  \dot x = -Lx ,
\end{align*}
where $x = [x_1\  \dots \ x_N]^T$ is the vector of all agent states and $L$ is the Laplacian of
$\commgraph$.  However, in order to be implemented, this control law
requires each agent to continuously have information about its
neighbors and continuously update its control law. 

Several recent works have been aimed at relaxing these 
requirements~\cite{XM-TC:13,CN-JC:14-acc,XM-LX-YCS-CN-GJP:15,GSS-DVD-KHJ:13}.
However, they all require agents to be ``listening'' continuously to their neighbors, i.e. when an event is triggered by one agent, its
neighbors are immediately aware and can take action accordingly.

Unfortunately, as we assume here that agents are unable to perform any communication while submerged,
we cannot continuously detect neighboring events that occur.
Instead, we assume that agents are only able to update their control
signals when their own events are triggered (i.e., when they are surfaced).  
Let $\{ t_i^\ell \}_{\ell \in \integernonnegative}$ be the sequence
of times at which agent~$i$ surfaces. Then, our algorithm is based on a piecewise 
constant implementation of the controller~\eqref{eq:continuouslaw}
given by
\begin{align}\label{eq:control}
u_i^\star(t) = - \sum_{j \in \NN_i} ( x_i(t_i^\ell) - x_j(t_i^\ell)), \quad t \in [t_i^\ell, t_i^{\ell+1} ).
\end{align}

\begin{remark}
{\rm Later we will allow the control input $u_i(t)$ to change in a limited way while agent $i$ is submerged, but for now we assume that the control is piecewise constant on the intervals $[t^\ell_i, t^{\ell+1}_i)$. Motivation for and details behind changing the control while submerged are discussed later in section~\ref{sec:exp}. \oprocend }
\end{remark}

The purpose of this paper is to develop a self-triggered algorithm that
determines how the the sequence of times~$\{ t_i^\ell \}$ and control inputs $u_i(t)$
can be chosen such that the system converges to the desired consensus statement. 
More specifically, each agent~$i$ at each surfacing time~$t_i^\ell$
must determine the next surfacing time~$t_i^{\ell+1}$ and control $u_i(t)$ only using information
available on the cloud at that instant. 
The closed loop system should then have trajectories
such that $|x_i(t) - x_j(t)| \rightarrow 0$ as $t \rightarrow \infty$ for all $i,j \in \until{N}$.
We describe the cloud communication model next.

\subsection{Cloud communication model}\label{se:comm}

We assume that there exists a base station or ``cloud'' that agents are able to upload data to and download data from when they are surfaced.
This cloud can store any finite amount of data, but can perform no computation. 
At any given time $t \in [t_i^\ell, t_i^{\ell+1})$, the cloud stores the following information about agent $i$:
the last time $\tlast_i(t) = t^\ell_i$ that agent $i$ surfaced, the next time $\tnext_i(t) = t_i^{\ell+1}$ that agent $i$ is scheduled to surface, 
the state $x_i(\tlast_i)$ of agent $i$ when it last surfaced, and the last control signal $u_i(\tlast_i)$ used by agent $i$.
The server also contains a control expiration time $\texp_i \leq \tnext_i$
and a \emph{promise} $M_i$ for each agent~$i$ which will be explained later in Section~\ref{sec:alg}. This information is summarized in Table~\ref{ta:communication}.

\begin{table}
\centering
\scriptsize
\begin{tabular}{|cl|}
\hline
$\timelast_i$ & Last time agent~$i$ surfaced \\
$\texp_i$ & Control expiration time of agent~$i$ \\
$\timenext_i$ & Next time agent~$i$ will surface \\
$x_i(\timelast_i)$ & Last updated position of agent~$i$ \\
$u_i(\timelast_i)$ & Last trajectory of agent~$i$ \\
$M_i(\timelast_i)$ & Most recent control promise from agent~$i$ \\
\hline
\end{tabular}
\caption{Data stored on the cloud for all agents~$i$ at any time~$t$.}\label{ta:communication}
\end{table}

For simplicity, we assume that agents can download/upload information to/from the cloud instantaneously. %
Let $t^\ell_i$ be a time at which agent $i$ surfaces to communicate with the cloud. 
The communication link is established at time $t^\ell_i$, and we immediately update $\tlast_i = t^\ell_i$ and $x_i(t_i^\ell)$ based on agent $i$'s current position.

While the link is open, agent $i$ downloads all the information in Table~\ref{ta:communication} for each neighbor $j\in\mathcal{N}_i$.
Using this information, agent $i$ (instantaneously) computes its control signal $u_i(t^\ell_i)$ and next surfacing time $t^{\ell+1}_i$ such that it knows it will make a net positive contribution to the consensus over the interval $[t^\ell_i, t^{\ell+1}_i)$. 
Finally, before closing the communication link and diving, agent $i$ calculates a promise $M_i$ bounding its future control inputs and uploads all data to the server.

\begin{remark}
{\rm Because of the existence of a centralized cloud server, it may be tempting to ask why the communication graph $\mathcal{G}$ is not always the complete graph $K_N$. 
Note that the amount of computation an agent does under our algorithm is quadratic in the number of neighbors $|\NN_i|$ (see Remark~\ref{rem:complexity}).
To ensure scalability for agents with limited computational capabilites as the number of agents in the network grows extremely large, then, it may be necessary to force a more limited communication topology.
Furthermore, especially with the increasing popularity of software defined networking, it is true that while any agent $i$ may be \emph{able} to communicate with any other agent $j$, it should be avoided whenever possible.  \oprocend }
\end{remark}

\begin{problem}
Given $N$ agents with dynamics~\eqref{eq:dynamics} and the communication model described in Section~\ref{se:comm}, for each agent~$i$, find an algorithm that prescribes when to communicate with the cloud based on currently available information and a control input $u_i(t)$ used in between communications $t \in [t_i^\ell, t_i^{\ell+1})$, such that  
\begin{align}
|x_j(t) - x_i(t)| \rightarrow 0
\end{align}
as $t \rightarrow \infty$ for all agents $i,j \in \{ 1, \dots, N \}$. 
\end{problem}

In the next section we describe this algorithm in detail.
\section{Distributed Trigger Design}
\label{sec:alg}
Consider the objective function 
\begin{align}
V(x(t)) &= \frac{1}{2} x^T(t) L x(t),
\end{align}
where $L$ is the Laplacian of the connected communication graph~$\mathcal{G}$. 
Note that $V(x) \geq 0$ for all $x \in \real^N$ and $V(x) = 0$ if and only if $x_i = x_j$ for all $i,j \in \{1, \dots, N\}$.
Thus, the function $V(x)$ encodes the objective of the problem and we are interested in driving $V(x) \rightarrow 0$. 
For simplicity, we drop the explicit dependence on time when referring to time~$t$. 

Taking the derivative of~$V$ with respect to time, we have
\begin{align}
\dot{V} &= \dot{x}^T L {x} 
= -\sum_{i=1}^N \dot{x}_i \sum_{j\in\NN_i}( x_j-x_i)
\end{align}

Let us split up $\dot{V} = \sum_{i=1}^N \dot{V}_i$, where
\begin{align}\label{eq:split}
\dot{V}_i \triangleq - \dot{x}_i \sum_{j\in\NN_i}( x_j-x_i).
\end{align}
Note that we have essentially distributed~$\dot{V}$ in a way that clearly shows how
each agent's motion directly contributes to the global objective, allowing us to write
\begin{align}
V(x(t)) %
&= V(x(0)) + \sum_{i=1}^N \int_0^t \dot{V}_i(x(\tau)) d\tau .
\end{align}

Ideally, we now wish to design a self triggered algorithm such that $\dot{V}_i(x(t)) \leq 0$ for all agents $i$ at all times $t$.
Thus at surfacing time $t^\ell_i$, agent $i$ must determine $t^{\ell+1}_i$ and $u_i(t)$ such that $\dot{V}_i(t) \leq 0$ for all $t \in [t_i^\ell, t^{\ell+1}_i)$. 

While in the fully developed algorithm we will allow an agent to modify its control while still submerged, for now we assume that the control input is constant on the entire submerged interval and defer the discussion of the control ``expiration time'' $\texp_i$; its motivation and (minor) modifications to the algorithm to section~\ref{sec:exp}.

Note that given the information agent $i$ downloaded from the server at time $t^\ell_i$, it is able to exactly compute the state of a neighboring agent $j\in\mathcal{N}_i$ up to the time it resurfaces $\tnext_j$. 
For any $t \in [t^\ell_i, \tnext_j]$, %
\begin{align}
x_j(t) = x_j(\tlast_j) + u_j(\tlast_j)(t - \tlast_j) .
\end{align}

At time $\tnext_j$, however, agent $j$ autonomously updates its control signal in a way unknown to agent $i$, making it difficult to determine how agent $i$ should move without surfacing. 
To remedy this, we borrow an idea of \emph{promises} from team-triggered control~\cite{CN-JC:16-tac}. 
Suppose that although we don't know $\dot{x}_j(t)$ exactly for $t > T_i$, we have access to some bound $M_j(t) > 0$ such that 
$|\dot{x}_j(t)| \leq M_j(t)$.

Using this information, we introduce the notion of agent $j$'s reachable set as determinable by agent $i$. 
For any $j \in \NN_i$, let $R^i_j(t)$ be the set of states at time $t$ within which agent $i$ can determine that agent $j$ must be in. 
For $t \leq \tnext_j$, agent $i$ is able to determine $x_j(t)$ exactly and so $R^i_j(t) = \{x_j(t)\}$ is a singleton containing agent j's exact position.
For $t > \tnext_j$, as all agent $i$ knows is a bound on agent $j$'s control law, $R^i_j$ is a ball that grows at a rate determined by agent $j$'s promise $M_j$:
\begin{align}
R^i_j(t) = \left\{ \begin{array}{cc}  \{x_j(\tlast_j) + u_j(\tlast_j)(t - \tlast_j) \} & t \leq \tnext_j , \\
                            B( x_j(\tnext_j), M_j(t) (t-\tnext_j) ) & \textrm{otherwise,} \end{array}\right.
\end{align}
where $B(x, r)$ is a closed ball of radius $r$ centered at $x$. 
$R^i_i$ is simply the singleton
\begin{align}
R^i_i(t) = \{ x_i(\tlast_i) + u_i(\tlast_i) (t-\tlast_i) \} .
\end{align}

We can now express the latest time that agent $i$ can still be sure it is contributing positively to the objective.
\begin{definition}
$T^\star_i$ is the first time $T^\star_i \geq t^\ell_i$ after which agent $i$ can no longer guarantee it is positively contributing to the objective, the solution to the following:
\begin{equation}
\begin{aligned}
& \underset{t \geq t^\ell_i}{\textnormal{infimum}}
& & t \\
& \textnormal{subject to}
& & \max_{x(t) \in R^i(t)} \dot{V}_i(x(t)) > 0 ,
\end{aligned}
\label{eq:t_opt}
\end{equation}
where $R^i(t)$ is defined as the set of all states $x_j(t)$ for $j \in \NN_i \cup \{ i \}$ such that each $x_j(t)$ satisfies $x_j(t) \in R^i_j(t)$. 
\end{definition}

It is easy to compute the solution to~\eqref{eq:t_opt} exactly given the structure of $R^i_j(t)$ given above. 
Let $\pi_t$ be a sorted ordering on the next surfacing times of all of agent $i$'s neighbors, i.e. let $\pi_t : [|\NN_i|] \rightarrow \NN_i$ be a one-to-one function such that $\tnext_{\pi_t(1)} \leq \tnext_{\pi_t(2)} \leq \ldots \leq \tnext_{\pi_t(|\NN_i|)}$.
We abuse notation slightly by additionally setting $\tnext_{\pi(0)} = \tlast_i$ and $\tnext_{\pi(|\NN_i| + 1)} = \infty$ so the union of the intervals $[\tnext_{\pi(k)}, \tnext_{\pi(k+1)})$ for $k \in \{0, 1, \ldots, |\NN_i|\}$ covers all $t \geq \tlast_i$. 

\begin{proposition}
\label{prop:tstar}
Let $\tau^{\star,(k)}_i$ be the solution to the following optimization:
\begin{equation}
\begin{aligned}
& \underset{t}{\textnormal{infimum}}
& & t \\
& \textnormal{subject to} 
  & & \tnext_{\pi(k)} \leq t \leq \tnext_{\pi(k+1)} , \\
 &&&\sum_{m'=1}^k \alpha_{i{\pi(m')}}(\tnext_{\pi(m')}) + \sum_{m=k+1}^{|\NN_i|} \alpha_{i{\pi(m)}}(\tlast_i)  \\
  &&& \ \ \ \  +    \sum_{m'=1}^k (t-\tnext_{\pi(m)}) \gamma_{i{\pi(m')}}(\tlast_i)  \\
   &&&\ \ \ \ + \sum_{m=k+1}^{|\NN_i|}  (t-\tnext_{\pi(m)})\beta_{i{\pi(m)}}(\tlast_i) > 0 ,
\label{eq:t_opt_constrained_prop}
\end{aligned}
\end{equation}
where 
\begin{gather}
\alpha_{ij}(t) \triangleq -u_i(t) (x_j(t) - x_i(t)) \label{eq:alpha}, \\
\beta_{ij}(t) \triangleq -u_i(t)(u_j(t) - u_i(t)) \label{eq:beta}, \\
\gamma_{ij}(t) \triangleq |u_i(t)| M_j(t) + u_i(t)^2 .
\label{eq:gamma}
\end{gather}
Then, the solution $T^\star_i$ to~\eqref{eq:t_opt} can be computed exactly as 
\begin{align}
  T^\star_i = \min \ \{ \tau^{\star,(k)}_i \ | \ k \in \{0, \ldots, |\NN_i|\}  \} .
\end{align}
\end{proposition}

\begin{proof}
See Appendix~\ref{sec:prop1_proof}.
\end{proof}

It is additionally possible to allow agent $i$ to remain submerged for longer by allowing $\dot{V}_i$ to temporarily become positive, as long as we select $t^{\ell+1}_i$ such that the total contribution to the objective $V$ on the interval $[t^\ell_i, t^{\ell+1}_i)$, 
\begin{align}
\Delta V_i^\ell \triangleq \int_{t^\ell_i}^{t^{\ell+1}_i}\dot{V}_i(\tau)d\tau, \label{eq:vdot_integral}
\end{align}
is nonpositive.

\begin{definition}
$T^{\textrm{total}}_i$ is the first time $T^{\textrm{total}}_i \geq t^\ell_i$ after which agent $i$ can no longer guarantee its total contribution over the submerged interval is positive, the solution to the following:
\begin{equation}
\begin{aligned}
& \underset{t \geq t^\ell_i}{\textnormal{infimum}}
& & t \\
& \textnormal{subject to}
& & \max_{x(t) \in R^i(t)} \int_{\tlast_i}^t \dot{V}_i(\tau)d\tau > 0 ,
\end{aligned}
\label{eq:t_int_opt}
\end{equation}
\end{definition}

To compute $T^{\textrm{total}}_i$, we follow a similar approach.

\begin{proposition}
\label{prop:ttotal}
Let $\tau^{tot,(k)}_i$ be the solution to the following optimization:
\begin{equation}
\begin{aligned}
& \underset{t}{\textnormal{infimum}}
& & t \\
& \textnormal{subject to}
& &\tnext_{\pi(k)} \leq t \leq \tnext_{\pi(k+1)} , \\
&&& \sum_{m' = 1}^k \Big[ \alpha_{i\pi(m')}(\tlast_i) (\tnext_{\pi(m')}-\tlast_i) + \frac{1}{2} \beta_{i{\pi(m')}}(\tlast_i) (\tnext_{\pi(m')}-\tlast_i)^2 \\
                     &&&\ + \alpha_{i{\pi(m')}}(\tnext_{\pi(m')})(t-\tnext_{\pi(m')}) + \frac{1}{2} \gamma_{i{\pi(m')}}(\tlast_i) (t-\tnext_{\pi(m')})^2 \Big] \\
  &&& + \sum_{m=k+1}^{|\NN_i|}  \Big[ \alpha_{i{\pi(m)}} (\tlast_i) (t-\tlast_i) + \frac{1}{2} \beta_{i{\pi(m)}}(\tlast_i) (t-\tlast_i)^2 \Big] > 0 .
\label{eq:vint_opt_constrained_prop}
\end{aligned}
\end{equation}
Then, the solution $T^{\textrm{total}}_i$ to~\eqref{eq:t_int_opt} can be computed as
\begin{align}
T^{\textrm{total}}_i = \min \ \{ \tau^{tot,(k)}_i \ | \ k \in \{ 0, \dots, |\NN_i| \} \}
\end{align}
\end{proposition}

\begin{proof}
See Appendix~\ref{sec:prop2_proof}.
\end{proof}

\begin{remark}
{\rm
Although the optimization constraints  in Propositions~\ref{prop:tstar} and~\ref{prop:ttotal} appear complex, note that they are linear or quadratic in $t$ and so the infimums can be solved for easily.
Consider a problem of the form
\begin{equation}
\begin{aligned}
& \underset{t}{\textnormal{infimum}}
& & t \\
& \textnormal{subject to}
& & g(t) > 0 , \\
&&& t_1 \leq t \leq t_2 ,
\label{eq:example_opt}
\end{aligned}
\end{equation}
where $g(t)$ is a polynomial in $t$. 

Let $r_1 \leq r_2 \leq \dots \leq r_K$ be the roots of $g$ that lie in the interval $(t_1, t_2)$, and let $r_0 = t_1$ and $r_{K+1} = t_2$. 
The solution $t^\star$ to~\eqref{eq:example_opt} is the smallest $r_i$, $i=0,\dots,K$, such that $g(r_i) \geq 0$ and $g(\frac{1}{2}(r_i + r_{i+1})) > 0$. 
If no such $r_i$ exists, $t^\star = \infty$.
}\oprocend
\end{remark}

\begin{remark}
{\rm
\label{rem:complexity}
The computation of the constraint coefficients in~\eqref{eq:vint_opt_constrained_prop} and~\eqref{eq:t_opt_constrained_prop} takes $O(|\NN_i|)$ time, and in both cases $|\NN_i| + 1$ such problems must be solved.
Thus, computation of $T^\star$ and $T^{\textrm{total}}$ both takes $O(|\NN_i|^2)$ time.
}\oprocend
\end{remark}

Selecting $t^{\ell+1} = T^\star_i$ ensures that $\dot{V_i} < 0$ over the submerged interval, ensuring that agent $i$ is making progress towards the global objective at all $t$.
Selecting $t^{\ell+1} = T^\textrm{total}_i$ introduces a trade-off; while this time allows the agent to remain submerged for longer, as it allows some positive contribution to the objective function, overall progress is slower.
Thus, we propose a tuning parameter $\sigma_i \in [0,1]$, selecting a time $t^{\ell+1}_i$ such that $T^\star_i \leq t^{\ell+1} \leq T^\textrm{total}_i$:
\begin{align}
  t^{\ell+1}_i = (1-\sigma_i) T^\star_i + \sigma_i T^\textrm{total}_i .
\end{align}
By continuity of $\dot{V}_i$ and the definitions of $T^\star_i$ and $T^\textrm{total}_i$, it is guaranteed for $t_i^{\ell+1} \in [ T^\star_i, T^\textrm{total}_i]$ that we still have $\Delta V_i^\ell \leq 0$ (as defined in~\eqref{eq:vdot_integral}) with this definition.
Setting all $\sigma_i$ near 0 allows faster convergence with more frequent surfacing, while $\sigma_i$ near 1 results in slower convergence but less frequent surfacing. 

\subsection{Selecting promises $M_j$}
\label{sec:promise}
As it isn't possible in general for agent $i$ to bound a neighbor agent $j$'s future control inputs from past state and control information, instead each agent makes a \emph{promise} $M_i$ about its future control inputs each time it connects to the server.
In the preceding section, we assumed that the bound $|\dot{x}_j(t)| \leq M_j(t)$ was satisfied at all times without describing how to make it so.
Here, we describe how to co-design the control laws $u_i(t)$ and promises $M_i(t)$ to ensure that this actually holds at all times.

Let $M_i^\ell$ be the promise made by agent $i$ at time $t^\ell_i$. 
From the constraints in Propositions~\ref{prop:tstar},~\ref{prop:ttotal}, it is clear that the smaller $M_j$ is for any $j \in \NN_i$, the longer agent $i$ is able to stay submerged. 
However, limiting the control too much below the ideal control \eqref{eq:control} will slow convergence.

We consider a promise rule in which at time $t_i^\ell$ agent $i$ sets its promise to be a function of $|u^\star_i(t^\ell_i)|$:
\begin{align}
  M^\ell_i &= f \left( \left\vert u^\star_i(t^\ell_i) \right\vert \right) .
\label{eq:promise}
\end{align}
For example, $f(x) = cx$ provides a parameter $c$ that effectively allows another trade-off between convergence speed and communication frequency.
Note however that this does not mean agent $i$ can use its ideal control law at all times; if the new desired input is greater in magnitude than a previous promise, to remain truthful to previous promises agent $i$ must wait until the new promise has been received by all of its neighbors when they surface before it can use its desired control input.

Let $\tau_{ij}^\ell$ be the time that agent $j$ sees agent $i$'s $\ell$th promise, i.e. $\tau_{ij}^\ell = \tnext_j(t^\ell_i)$. 
When submerging for an interval $[t^\ell_i, t^{\ell+1}_i)$, agent $i$ needs to guarantee that all promises $M_i$ currently believed by $j\in\NN_i$ are abided by.

Let $\plast_{ij}(t)$ be the most recent promise by agent $i$ that agent $j$ is aware of at time $t$, i.e. 
\begin{align}
  \plast_{ij}(t) &= \underset{\ell\ : \ \tau^\ell_{ij} \leq t}{\arg \max} \ \tau_{ij}^\ell ,
\end{align}
and let $\mathcal{P}_i^\ell$ be the set of promise indices that agent $i$ must abide by when submerging on $[t^\ell_i, t^{\ell+1}_i)$, i.e.
\begin{align}
  \mathcal{P}^\ell_i &= \left\{\ \plast_{ij}(t) \ |\ j \in \NN_i, t \in [t^\ell_i, t^{\ell+1}_i) \ \right\} .
\end{align}

To abide by all promises that agent $i$'s neighbors believe about its controls, then, it simply needs to bound its control input magnitude by 
\begin{align}
  u^\textrm{max}_i(t^\ell_i) = \min_{k \in \mathcal{P}^\ell_i} M^k_i. \label{eq:umax}
\end{align}

With this bound, the actual control law used and uploaded by agent $i$ on the interval $[t^\ell_i, t^{\ell+1}_i)$ is given by bounding the ideal control magnitude by $u^\textrm{max}_i(t^\ell_i)$, or
\begin{align}
\label{eq:realcontrol}
u_i(t^\ell_i) &= \left\{ \begin{array}{ll} u^\star_i(t^\ell_i) & \left\vert u^\star_i(t^\ell_i) \right\vert \leq u^\textrm{max}_i(t^\ell_i) ,  \\
                                                          u^\textrm{max}_i(t^\ell_i) \frac{u^\star_i(t^\ell_i)}{\left\vert u^\star_i(t^\ell_i)\right\vert} & \textrm{otherwise.} \end{array} \right.
\end{align}

\subsection{Maximum submerged time}
\label{sec:max_time}
The method presented thus far is almost complete; however, consider the case in which a subset of the communication graph has locally reached ``consensus'' while the system as a whole has not.
If there is some agent $i$ such that at agent $i$'s next surfacing time $t^\ell_i$ we have $x_i(t^\ell_i) = x_j(t^\ell_i)$ for all $j \in \NN_i$, and furthermore that $u_j(t^\ell_i) = 0$ for all $j \in \NN_i$, then it would set its next triggering time to infinity. 
 An examination of the constraints in~\eqref{eq:t_opt_constrained_prop} and~\eqref{eq:vint_opt_constrained_prop} then reveals that the feasible set in both cases is equal to the empty set~$\emptyset$. The times $T^\star_i$ and $T^{\textrm{total}}_i$ will thus be chosen as $\inf \emptyset = \infty$.

In order to guarantee consensus in all situations, and as it is impossible for agent $i$ to obtain information outside of its immediate neighbors, it is necessary to introduce a maximum submerged time $\tmax$. 
If an agent $i$ computes a $\tideal_i$ such that $\tideal_i - \tlast_i > \tmax$, the agent instead chooses $\tnext_i = \tlast_i + \tmax$.
This ensures that despite a region of the communication network being at local ``consensus,'' no agent will effectively remove itself from the system and information will continue to propagate.

\subsection{Avoiding Zeno behavior}
\label{sec:exp}
While the presented method of selecting surfacing times guarantees convergence, it is susceptible to Zeno behavior, i.e. requiring some agent $i$ to surface an infinite number of times in a finite time period. 
To avoid this behavior, we introduce a fixed dwell time $T^\textrm{dwell}_i > 0$, and force each agent to remain submerged for at least a duration of $\tdwell_i$. Unfortunately, this means that
in general, there may be times at which an agent~$i$ is forced to remain submerged even when it does not know how to move to contribute positively to the global task (or it may not even be possible if it is at a local minimum).
Remarkably, from the way have have distributed~$\dot{V}$ using~\eqref{eq:split}, if agent $i$ sets $u_i(t) = 0$, its instantaneous contribution to the global objective is exactly 0. 

Thus, we allow an agent's control to change while it is submerged and modify the control law described in the previous section as follows.
If the chosen ideal surfacing time $\tideal_i = (1-\sigma_i)T^\star_i+\sigma_i T_i^\textrm{total}$ is greater than or equal to $t^\ell_i + \tdwell_i$, then nothing changes; agent $i$ sets its next surfacing time $t^{\ell+1}_i = \tideal_i$ and uses the control law \eqref{eq:realcontrol} on the entire submerged interval.

If, on the other hand, $\tideal_i  < t^\ell_i + \tdwell_i$, we let agent $i$ use the usual control law until $\tideal_i$, until which it knows it can make a positive contribution. 
After $\tideal_i$, agent $i$ no longer is certain it can make a positive contribution to the global objective.
Thus, we force agent $i$ to remain still until it has been submerged for a dwell time duration. 
In other words, we set $\tnext = t^\ell_i + \tdwell_i$, $\texp =\tideal_i$ and use control law \eqref{eq:realcontrol} on the interval $[t^\ell_i, \texp_i)$.

For $t \in [\texp_i, t^{\ell+1}_i)$, we then set $u_i(t) = 0$, and note that because $\dot{V}_i(t) = 0$ on this interval we still have the desired contribution to the global objective
\begin{align}
 \int_{t^\ell_i}^{t^{\ell+1}_i}\dot{V}_i(\tau)d\tau %
  &= \int_{t^\ell_i}^{\texp_i}\dot{V}_i(\tau)d\tau < 0.
\end{align}

Agent $i$ is then still able to calculate the position of any neighbor $j$ exactly for any $t < \tnext_j$ using information available on the cloud by 
\begin{align}
\label{eq:neighbors}
x_j(t) = \left\{ \begin{array}{ll} x_j(\tlast_j) + u_j(\tlast_j)(t - \tlast_j) & t < \texp_j , \\ 
                                              x_j(\tlast_j) + u_j(\tlast_j)(\texp_j - \tlast_j) & \textrm{otherwise.}\end{array} \right.
\end{align}

An overview of the fully synthesized self-triggered coordination algorithm is presented in Algorithm~1.
Next, we present the main convergence result of this algorithm.

\begin{algorithm}[t]
{\footnotesize
  At surfacing time $t_i^\ell$, 
  agent $i \in \until{N}$  performs:
  \begin{algorithmic}[1]
    \State download $\timelast_j, \texp_j, \timenext_j, x_j(\timelast_j), u_i(\timelast_j), M_j$ for all $j \in \NN_i$ from cloud
    \State compute neighbor positions $x_j(t_i^\ell)$ using \eqref{eq:neighbors}
    \State compute ideal control $u_i^\star(t^\ell_i) = - \sum_{j \in \NN_i} (x_i(t_i^\ell) - x_j(t_i^\ell))$
    \State compute $u^\textrm{max}_i(t^\ell_i)$ using \eqref{eq:umax} and saved $\tau_{ij}$ data
    \State compute control $u_i(t_i^\ell)$ with \eqref{eq:realcontrol}
    \State compute $T_i^\star$ as the solution to~\eqref{eq:t_opt}
    \State compute $T^{\textrm{total}}$ as the solution to~\eqref{eq:t_int_opt}
    \State set $t^{\textrm{ideal}}_i = (1 -\sigma_i)T_i^\star + \sigma_i T^{\textrm{total}}_i$
    \If {$\tideal_i > t^\ell_i + \tmax$}
      \State set $\texp_i = t^{\ell+1}_i = t^\ell_i + \tmax$
    \ElsIf {$\tideal_i < t^\ell_i + \tdwell_i$}
      \State set $\texp_i = \tideal_i$
      \State set $t^{\ell+1}_i = t^\ell_i + \tdwell_i$
    \Else
      \State set $\texp_i = t^{\ell+1}_i = \tideal_i$
    \EndIf
    \State upload promise $M_i = \left\vert u^\star_i(t^\ell_i) \right\vert$ to cloud
    \State upload $\timelast_i = t_i^\ell$, $\timenext_i = t_i^{\ell+1}$, $\texp_i$, $u_i(t_i^\ell)$,  $x_i(t_i^\ell)$ to cloud
    \State dive and set $u_i(t) = u_i(t_i^\ell)$ for $t \in [t_i^\ell, \texp_i)$, $u_i(t) = 0$ for $t \in [\texp_i, t_i^{\ell+1})$
  \end{algorithmic}}
\caption{\hspace*{-.5ex}: \small \algoFull}\label{tab:full}
\end{algorithm}

\begin{theorem}
Given the dynamics \eqref{eq:dynamics} and $\mathcal{G}$ connected, if the sequence of times $\{t_i^\ell\}$ and control laws $u_i(t^\ell_i)$ are determined by Algorithm 1 for all $i \in \{1, \dots, N\}$, then 
\begin{align}
  |x_i(t) - x_j(t)|  \rightarrow 0
\end{align}
for all $i,j\in\{1, \dots, N\}$ as $t \rightarrow \infty$.
\end{theorem}
\begin{proof}
First, because $t^{\ell+1}-t^\ell \geq \tdwell_i > 0$, Zeno behavior is impossible, and so $x(t)$ exists for all $t \geq 0$.

Now, consider the objective function $V = \frac{1}{2} x^T L x$. Then, recall we can decompose it as
\begin{align}
V(x(t)) %
  &= V(x(0)) + \sum_{i=1}^N \int_0^t \dot{V}_i(\tau)d\tau.
\end{align}

Letting $\ell^\textrm{max}_i(t) = \argmax_{\ell\in\mathbb{Z}_{\geq 0}} t^\ell_i \leq t$ be the index such that $\tlast_i(t) = t^{\ell^\textrm{max}_i(t)}_i$, we can further expand this as
\begin{align}
V(x(t)) &= V(x(0)) + \sum_{i=1}^N \sum_{\ell=0}^{\ell^\textrm{max}_i(t)} \int_{t^\ell_i}^{\min\{t^{\ell+1}_i, t\}} \dot{V}_i(\tau)d\tau.
\end{align}

Consider $\Delta V_i^\ell$ (as defined in \eqref{eq:vdot_integral}), which is
the net contribution of agent~$i$ over the time interval $[t_i^\ell, t_i^{\ell+1})$.
Note that we have explicitly designed algorithm 1 to ensure that each $\Delta V_i^\ell \leq 0$.
 $T^\textrm{total}_i$ is exactly the earliest time at which we can no longer guarantee $\Delta V_i^\ell \leq 0$ in the worst case, and we always have $t^\textrm{ideal} < T^\textrm{total}_i$. 
Furthermore, if the dwell time forces an agent to stay submerged past $t^\textrm{ideal}_i$, by setting the control to 0 we ensure that the total contribution on the submerged interval is equal to the contribution from $t^\ell_i$ to $t^\textrm{ideal}_i$.
Thus, we know that $\Delta V_i^\ell \leq 0$ for all $i \in \{1,\dots,N\}$ and for all $\ell \in \{1, \dots, \ell^\textrm{max}_i - 1\}$.

Thus, we have
\begin{align}
  V(x(t)) \leq V(x(0)) + \sum_{i=1}^N \sum_{\ell=0}^{\ell^\textrm{max}_i - 1} \Delta V^\ell_i.
\end{align}

It is clear that $V(x(t))$ is a nonincreasing function along the system trajectories and bounded below by 0. Furthermore, by the introduction of the dwell time it is clear that each sequence of times $\{ t_i^\ell \}_{\ell \in \integernonnegative}$ goes to infinity as $\ell \rightarrow \infty$. Thus, $\lim_{t\rightarrow\infty} V(x(t)) = C \geq 0$ exists.

Because $\Delta V_i^\ell \leq 0$ for all $\ell$ and $V(x)$ is bounded from below, it is guaranteed that $\Delta V_i^\ell \rightarrow 0$ as $t \rightarrow \infty$.
Thus, by LaSalle's Invariance Principle~\cite{khalil2002nonlinear}, the trajectories of the system converge to the largest invariant set contained in 
\begin{align}
\{x\in\mathbb{R}^N \ | \ \dot{V}_i(x) = 0 \ \forall i \in \{1, \dots, N\} \} .
\end{align}

From examination of the local objective contribution~\eqref{eq:split}, we see that $\dot{V}_i(x) = 0$ if and only if either $u_i(t) = 0$ or $\sum_{j \in \NN_i} x_j-x_i = 0$.
First, note that 
$
  \sum_{j \in \NN_i} x_j - x_i = 0 
$
for all $i$ if and only if the system is at consensus.
This condition is equivalent to $Lx = 0$, and as we assume $\mathcal{G}$ is connected, $\ker(L) = \{\mathbf{1}_N\}$.

Now, assume the system is not at consensus, so there is at least one agent $i$ with $\sum_{j \in \NN_i} x_j - x_i \neq 0$.
From the control law~\eqref{eq:control} it is clear that this implies the next time agent $i$ surfaces, $t^\ell_i$, we will have $u^\star_i(t^{\ell}_i) \neq 0$ as well. 
Thus, we simply have to prove that the next time $t^\ell_i$ that agent $i$ surfaces, it computes a $\tideal_i > t^\ell_i$ so the real control $u_i(t) \neq 0$ for a nonzero period of time. 
As $\tideal_i > T^\star_i$, the last time at which we can guarantee $\dot{V}_i(t) \leq 0$, it suffices to show $T^\star_i > t^\ell_i$.

$T^\star_i$ is computed as the earliest time after which our bound on $\dot{V}_i(t)$ is positive. 
From~\eqref{eq:Vij_bound_1} we can write this bound at time $t^\ell_i$ as
\begin{align}
\dot{V}_i(t^\ell_i) = -\left( \sum_{j \in \NN_i} x_j - x_i \right)^2 ,
\end{align}
which is strictly negative.
As the bound is a continuous function of time, this implies that the smallest $t$ such that we can no longer guarantee $\dot{V}_i(t) \leq 0$ is strictly greater than $t^\ell_i$. 
Thus, the next time agent $i$ surfaces, it will apply a nonzero control for a positive duration.

Finally, due to the existence of the maximum submerged time $\tmax$, we know that there exists a finite future time at which agent $i$ will surface, which completes the proof. 

\end{proof}
\section{Simulation}

In this section we simulate a system of 5 agents with initial condition $x = [9\ -2\ 0.5\ 8.5\ 4]^T$ for a total time of 10 seconds, and with all $\sigma_i = \sigma$ the same value.
In all simulations we set $\tdwell = 10^{-8}$ seconds, but the dwell time condition was never used. 
Similarly, we set $\tmax = 5$ seconds, but it never affected the simulation. 
The topology of the communication network is shown in Figure~\ref{fig:network}. 
We compare our algorithm presented here with the algorithm proposed in~\cite{bowman_cdc16}, as well as with a simple periodic triggering rule where each agent surfaces every $T$ seconds and uses the constant control law $u^\star_i(t^\ell_i)$ on each submerged interval. 
Note that for the undirected graph in Figure~\ref{fig:network} the system will converge as long as $T < T^* = 2/\lambda_\text{max}(L) = 0.4331$~\cite{xie_acc09}. 

\begin{figure}
  \centering
  {\includegraphics[width=.4\columnwidth]{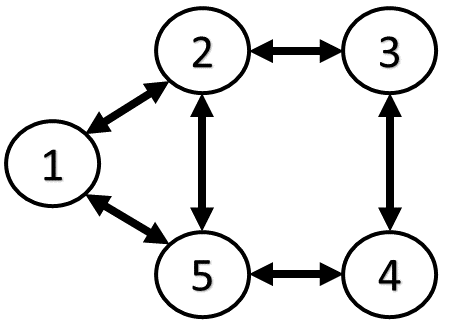}}
  \caption{Simulated communication network}\label{fig:network}
\end{figure}

We first compare our algorithm with $\sigma = 0.75$ and the promise function~\eqref{eq:promise} as $f(x) = x$, the algorithm presented in~\cite{bowman_cdc16} with $\sigma = 0.5$, and a periodic time-triggering rule with $T = 0.35$. 
The total number of surfacings by any agent up to time $t$, denoted $N_S(t)$, is shown in Figure~\ref{fig:surfacings}.

\newcommand{\figfactor}{0.35}

\begin{figure*}[htp]
  \centering
  \hspace*{\fill}%
  \subfigure[]{
%
%
\definecolor{mycolor1}{rgb}{0.00000,0.44700,0.74100}%
\definecolor{mycolor2}{rgb}{0.85000,0.32500,0.09800}%
\definecolor{mycolor3}{rgb}{0.92900,0.69400,0.12500}%
\begin{tikzpicture}

\begin{axis}[%
width=\figfactor\columnwidth,
scale only axis,
xmin=0,
xmax=6,
xlabel={$t$},
ymin=0,
ymax=90,
ylabel={$N_S(t)$},
axis background/.style={fill=white},
axis x line*=bottom,
axis y line*=left
]
\addplot[const plot,color=mycolor1,dashdotted,forget plot] plot table[row sep=crcr] {%
0	0\\
0.1	1\\
0.2	2\\
0.3	3\\
0.4	4\\
0.40511611582219	5\\
0.406829386058367	6\\
0.5	7\\
0.54820514503628	8\\
0.680035939441453	9\\
0.75373109628395	10\\
0.861946604491123	11\\
0.920408979513267	12\\
0.981538878852469	13\\
1.14510611577818	14\\
1.24088819315835	15\\
1.26652399459356	16\\
1.33388063438386	17\\
1.34088819315835	18\\
1.46983432488205	19\\
1.51009482720638	20\\
1.52757725857802	21\\
1.61009482720638	22\\
1.71009482720638	23\\
1.81009482720638	24\\
1.85037003760845	25\\
1.90243795478607	26\\
1.91009482720638	27\\
1.91737382166327	28\\
2.01009482720638	29\\
2.02152615784923	30\\
2.02365746328227	31\\
2.10713916558552	32\\
2.11009482720638	33\\
2.25125387415539	34\\
2.37114193383975	35\\
2.56192925822019	36\\
2.56423537379815	37\\
2.67200857319866	38\\
2.73474738346454	39\\
2.80039499720501	40\\
2.90335989951775	41\\
2.99944473441755	42\\
3.03738944274083	43\\
3.09944473441755	44\\
3.17253916439667	45\\
3.24116782786891	46\\
3.32992253563815	47\\
3.34116782786891	48\\
3.53039244085071	49\\
3.55084537956738	50\\
3.5745447315459	51\\
3.63039244085071	52\\
3.74404115635392	53\\
3.75033991775189	54\\
3.79407767697239	55\\
3.82629241867386	56\\
4.02515427877843	57\\
4.04689407620471	58\\
4.05197820480643	59\\
4.15197820480643	60\\
4.19313554189693	61\\
4.21359368626416	62\\
4.22556531595972	63\\
4.25197820480643	64\\
4.32556531595972	65\\
4.39141991016488	66\\
4.44460566925198	67\\
4.52241363347383	68\\
4.54178354117286	69\\
4.55174063890797	70\\
4.65174063890797	71\\
4.7114595583505	72\\
4.7160569200698	73\\
4.75273470855218	74\\
4.91452337993007	75\\
4.95758080135324	76\\
5.05487884462056	77\\
5.1146579953356	78\\
5.19219062318392	79\\
5.34695005341996	80\\
5.36419219937042	81\\
5.44695005341996	82\\
5.51814719929091	83\\
5.55289629182413	84\\
5.69738930538533	85\\
5.76194765905942	86\\
5.7990424469735	87\\
5.90404256418975	88\\
5.92015657055553	89\\
5.98420065875442	90\\
6.26588855016075	91\\
};
\addplot[const plot,color=mycolor2,dashed,forget plot] plot table[row sep=crcr] {%
0	0\\
1e-08	1\\
2e-08	2\\
3e-08	3\\
4e-08	4\\
5e-08	5\\
0.35000001	6\\
0.35000002	7\\
0.35000003	8\\
0.35000004	9\\
0.35000005	10\\
0.70000001	11\\
0.70000002	12\\
0.70000003	13\\
0.70000004	14\\
0.70000005	15\\
1.05000001	16\\
1.05000002	17\\
1.05000003	18\\
1.05000004	19\\
1.05000005	20\\
1.40000001	21\\
1.40000002	22\\
1.40000003	23\\
1.40000004	24\\
1.40000005	25\\
1.75000001	26\\
1.75000002	27\\
1.75000003	28\\
1.75000004	29\\
1.75000005	30\\
2.10000001	31\\
2.10000002	32\\
2.10000003	33\\
2.10000004	34\\
2.10000005	35\\
2.45000001	36\\
2.45000002	37\\
2.45000003	38\\
2.45000004	39\\
2.45000005	40\\
2.80000001	41\\
2.80000002	42\\
2.80000003	43\\
2.80000004	44\\
2.80000005	45\\
3.15000001	46\\
3.15000002	47\\
3.15000003	48\\
3.15000004	49\\
3.15000005	50\\
3.50000001	51\\
3.50000002	52\\
3.50000003	53\\
3.50000004	54\\
3.50000005	55\\
3.85000001	56\\
3.85000002	57\\
3.85000003	58\\
3.85000004	59\\
3.85000005	60\\
4.20000001	61\\
4.20000002	62\\
4.20000003	63\\
4.20000004	64\\
4.20000005	65\\
4.55000001	66\\
4.55000002	67\\
4.55000003	68\\
4.55000004	69\\
4.55000005	70\\
4.90000001	71\\
4.90000002	72\\
4.90000003	73\\
4.90000004	74\\
4.90000005	75\\
5.25000001	76\\
5.25000002	77\\
5.25000003	78\\
5.25000004	79\\
5.25000005	80\\
5.60000001	81\\
5.60000002	82\\
5.60000003	83\\
5.60000004	84\\
5.60000005	85\\
5.95000001	86\\
5.95000002	87\\
5.95000003	88\\
5.95000004	89\\
5.95000005	90\\
6.30000001	91\\
};
\addplot[const plot,color=mycolor3,solid,forget plot] plot table[row sep=crcr] {%
0	0\\
0.1	1\\
0.2	2\\
0.3	3\\
0.4	4\\
0.43193359375	5\\
0.5	6\\
0.5716796875	7\\
0.73984375	8\\
0.795703125	9\\
0.86025390625	10\\
1.104296875	11\\
1.2525390625	12\\
1.4833984375	13\\
1.5330078125	14\\
1.53671875	15\\
1.70888671875	16\\
1.98603515625	17\\
2.03642578125	18\\
2.17138671875	19\\
2.398046875	20\\
2.50751953125	21\\
2.77177734375	22\\
2.9091796875	23\\
2.9884765625	24\\
2.99150390625	25\\
3.2416015625	26\\
3.33798828125	27\\
3.34423828125	28\\
3.52353515625	29\\
3.55400390625	30\\
3.73408203125	31\\
3.75185546875	32\\
3.84208984375	33\\
3.9029296875	34\\
4.1248046875	35\\
4.2736328125	36\\
4.28662109375	37\\
4.32431640625	38\\
4.33046875	39\\
4.5140625	40\\
4.554296875	41\\
4.6212890625	42\\
4.67998046875	43\\
4.8205078125	44\\
5.01943359375	45\\
5.15146484375	46\\
5.15556640625	47\\
5.212890625	48\\
5.25302734375	49\\
5.39833984375	50\\
5.64921875	51\\
5.88037109375	52\\
6.02265625	53\\
};
\end{axis}
\end{tikzpicture}%
\label{fig:surfacings}}
  \hfill
  \subfigure[]{
%
%
\definecolor{mycolor1}{rgb}{0.00000,0.44700,0.74100}%
\definecolor{mycolor2}{rgb}{0.85000,0.32500,0.09800}%
\definecolor{mycolor3}{rgb}{0.92900,0.69400,0.12500}%
\begin{tikzpicture}

\begin{axis}[%
width=\figfactor\columnwidth,
scale only axis,
xmin=0,
xmax=6,
xlabel={$t$},
ymode=log,
ymin=1e-08,
ymax=10000,
yminorticks=true,
ylabel={$V(x(t))$},
axis background/.style={fill=white}
]
\addplot [color=mycolor1,dashdotted,forget plot]
  table[row sep=crcr]{%
0	135.125\\
0.1	60.67875\\
0.2	23.170625\\
0.3	15.938209375\\
0.4	17.1620875\\
0.40511611582219	16.888654269294\\
0.406829386058367	16.768086888037\\
0.5	12.6406575704345\\
0.54820514503628	10.7196408826976\\
0.680035939441453	8.02284785532103\\
0.75373109628395	7.13371544512131\\
0.861946604491123	6.14672183554307\\
0.920408979513267	5.63059524422274\\
0.981538878852469	5.12085808140932\\
1.14510611577818	3.90745248691982\\
1.24088819315835	3.28105846341424\\
1.26652399459356	3.12438387586312\\
1.33388063438386	2.73293987628213\\
1.34088819315835	2.69383433730487\\
1.46983432488205	2.02750716227466\\
1.51009482720638	1.84070296687084\\
1.52757725857802	1.76277784237409\\
1.61009482720638	1.43279342637453\\
1.71009482720638	1.08458291306849\\
1.81009482720638	0.793008523412318\\
1.85037003760845	0.691564687203133\\
1.90243795478607	0.576922724854207\\
1.91009482720638	0.563192996274462\\
1.91737382166327	0.550324434670837\\
2.01009482720638	0.375761634944367\\
2.02152615784923	0.357055645800482\\
2.02365746328227	0.353759218829058\\
2.10713916558552	0.233313977256913\\
2.11009482720638	0.229704700274617\\
2.25125387415539	0.09934297086101\\
2.37114193383975	0.0461740344307611\\
2.56192925822019	0.0653964720030734\\
2.56423537379815	0.0646446646266446\\
2.67200857319866	0.0346128497241524\\
2.73474738346454	0.0231693034905347\\
2.80039499720501	0.0129268978362696\\
2.90335989951775	0.00402163336693215\\
2.99944473441755	0.00266425813277939\\
3.03738944274083	0.00296046294873286\\
3.09944473441755	0.00474578399007401\\
3.17253916439667	0.00882534873944729\\
3.24116782786891	0.00640520189897202\\
3.32992253563815	0.00391638106288643\\
3.34116782786891	0.00365971010515167\\
3.53039244085071	0.0013196449042736\\
3.55084537956738	0.00122146586572736\\
3.5745447315459	0.00101111966857977\\
3.63039244085071	0.000802030969572227\\
3.74404115635392	0.000473949219079291\\
3.75033991775189	0.000460988541491923\\
3.79407767697239	0.000353670010754531\\
3.82629241867386	0.000300367037554961\\
4.02515427877843	8.06239528414396e-05\\
4.04689407620471	6.69401831317491e-05\\
4.05197820480643	6.47744200208255e-05\\
4.15197820480643	3.65681532001148e-05\\
4.19313554189693	2.99426896382274e-05\\
4.21359368626416	2.77306405208795e-05\\
4.22556531595972	2.67692280087302e-05\\
4.25197820480643	2.47994279959085e-05\\
4.32556531595972	1.97635855345721e-05\\
4.39141991016488	1.58635847568554e-05\\
4.44460566925198	1.3892389098563e-05\\
4.52241363347383	1.12614067478475e-05\\
4.54178354117286	1.0652519244354e-05\\
4.55174063890797	1.0167585514669e-05\\
4.65174063890797	6.10257918435327e-06\\
4.7114595583505	4.37303086515744e-06\\
4.7160569200698	4.27256365702518e-06\\
4.75273470855218	3.53895880251306e-06\\
4.91452337993007	2.22620834688513e-06\\
4.95758080135324	2.50995892012835e-06\\
5.05487884462056	1.8908253391675e-06\\
5.1146579953356	1.59360630434932e-06\\
5.19219062318392	1.2986010235179e-06\\
5.34695005341996	8.29461767884064e-07\\
5.36419219937042	8.1311745313002e-07\\
5.44695005341996	6.31887489879615e-07\\
5.51814719929091	5.07694875419209e-07\\
5.55289629182413	4.59397074490222e-07\\
5.69738930538533	3.06232513923744e-07\\
5.76194765905942	2.68086874945488e-07\\
5.7990424469735	2.54594547622168e-07\\
5.90404256418975	2.21228121622405e-07\\
5.92015657055553	2.16738605827256e-07\\
5.98420065875442	2.00555435632931e-07\\
6.26588855016075	1.51218441001824e-07\\
};
\addplot [color=mycolor2,dashed,forget plot]
  table[row sep=crcr]{%
0	135.125\\
1e-08	135.124990795\\
2e-08	135.124981590001\\
3e-08	135.124972385002\\
4e-08	135.124963180004\\
5e-08	135.124953975006\\
0.35000001	28.5959261408402\\
0.35000002	28.5959292726828\\
0.35000003	28.5959290265013\\
0.35000004	28.5959280517952\\
0.35000005	28.5959259898146\\
0.70000001	9.57244700253017\\
0.70000002	9.57244819352486\\
0.70000003	9.57244852873857\\
0.70000004	9.57244860068168\\
0.70000005	9.5724480952288\\
1.05000001	3.47403140842779\\
1.05000002	3.47403187667491\\
1.05000003	3.47403202601335\\
1.05000004	3.47403205112807\\
1.05000005	3.47403183150186\\
1.40000001	1.28496139058985\\
1.40000002	1.28496156809584\\
1.40000003	1.2849616385364\\
1.40000004	1.28496167287131\\
1.40000005	1.28496160308534\\
1.75000001	0.479339793574618\\
1.75000002	0.479339861565239\\
1.75000003	0.479339887602672\\
1.75000004	0.479339896688475\\
1.75000005	0.479339867224253\\
2.10000001	0.179773849282269\\
2.10000002	0.17977387517591\\
2.10000003	0.179773885982017\\
2.10000004	0.179773891455677\\
2.10000005	0.179773881469909\\
2.45000001	0.0676746887364886\\
2.45000002	0.0676746986022083\\
2.45000003	0.0676747025464441\\
2.45000004	0.0676747041204662\\
2.45000005	0.0676747000079185\\
2.80000001	0.0255429321727457\\
2.80000002	0.025542935926603\\
2.80000003	0.0255429375037578\\
2.80000004	0.0255429382908096\\
2.80000005	0.0255429368455284\\
3.15000001	0.00965888856416999\\
3.15000002	0.00965888999198\\
3.15000003	0.00965889057194491\\
3.15000004	0.00965889081655349\\
3.15000005	0.00965889023415087\\
3.50000001	0.00365728058790529\\
3.50000002	0.00365728113073086\\
3.50000003	0.00365728135851748\\
3.50000004	0.00365728147004624\\
3.50000005	0.00365728126041323\\
3.85000001	0.00138610239939832\\
3.85000002	0.00138610260572108\\
3.85000003	0.00138610269019773\\
3.85000004	0.0013861027269445\\
3.85000005	0.00138610264387333\\
4.20000001	0.000525676443232585\\
4.20000002	0.000525676521643786\\
4.20000003	0.000525676554462499\\
4.20000004	0.000525676570286224\\
4.20000005	0.000525676539885951\\
4.55000001	0.000199454205569664\\
4.55000002	0.000199454235360507\\
4.55000003	0.000199454247615811\\
4.55000004	0.00019945425304764\\
4.55000005	0.000199454241158081\\
4.90000001	7.57023920306792e-05\\
4.90000002	7.57024033453783e-05\\
4.90000003	7.57024080763821e-05\\
4.90000004	7.57024103283793e-05\\
4.90000005	7.57024059260605e-05\\
5.25000001	2.87392631238542e-05\\
5.25000002	2.87392674227782e-05\\
5.25000003	2.87392691996138e-05\\
5.25000004	2.87392699939628e-05\\
5.25000005	2.87392682921141e-05\\
5.60000001	1.09121831959116e-05\\
5.60000002	1.09121848293325e-05\\
5.60000003	1.09121855052519e-05\\
5.60000004	1.09121858340054e-05\\
5.60000005	1.09121851969473e-05\\
5.95000001	4.14378133331735e-06\\
5.95000002	4.14378195647165e-06\\
5.95000003	4.1437822085634e-06\\
5.95000004	4.14378232924898e-06\\
5.95000005	4.14378208361908e-06\\
6.30000001	1.57368038163725e-06\\
};
\addplot [color=mycolor3,solid,forget plot]
  table[row sep=crcr]{%
0	135.125\\
0.1	60.67875\\
0.2	23.170625\\
0.3	15.938209375\\
0.4	17.1620875\\
0.43193359375	15.6254274778775\\
0.5	11.9757097241789\\
0.5716796875	9.12833835828197\\
0.73984375	7.94571319931764\\
0.795703125	7.29033065673603\\
0.86025390625	6.42906996272308\\
1.104296875	4.10599474176634\\
1.2525390625	3.21587646751061\\
1.4833984375	1.95124938310368\\
1.5330078125	1.7594227070401\\
1.53671875	1.74541632093954\\
1.70888671875	1.10466355762484\\
1.98603515625	0.426557564774064\\
2.03642578125	0.346266994543607\\
2.17138671875	0.208219963526346\\
2.398046875	0.0779952819923297\\
2.50751953125	0.0616568462974229\\
2.77177734375	0.0211724123398144\\
2.9091796875	0.0106027031015892\\
2.9884765625	0.0082653413706619\\
2.99150390625	0.00814102181731084\\
3.2416015625	0.00299579138229799\\
3.33798828125	0.00181178372772009\\
3.34423828125	0.00175112764437313\\
3.52353515625	0.000689116565798661\\
3.55400390625	0.000595782855815324\\
3.73408203125	0.000499865472380671\\
3.75185546875	0.000499261947752058\\
3.84208984375	0.000392762139859344\\
3.9029296875	0.000290802668819166\\
4.1248046875	0.000103473126296758\\
4.2736328125	6.35526929927497e-05\\
4.28662109375	5.9365573436115e-05\\
4.32431640625	5.07008449515418e-05\\
4.33046875	4.94448030356337e-05\\
4.5140625	1.80929556622585e-05\\
4.554296875	1.43114392736569e-05\\
4.6212890625	1.03904084680509e-05\\
4.67998046875	9.36923983506019e-06\\
4.8205078125	5.63711171036336e-06\\
5.01943359375	3.07283838033088e-06\\
5.15146484375	3.13023706169629e-06\\
5.15556640625	3.15450132137591e-06\\
5.212890625	3.41168200757589e-06\\
5.25302734375	3.0950814833273e-06\\
5.39833984375	2.14289935175768e-06\\
5.64921875	1.29970312120356e-06\\
5.88037109375	1.09493165754765e-06\\
6.02265625	7.90769515650281e-07\\
};
\end{axis}
\end{tikzpicture}%
\label{fig:lyap}}
  \hspace*{\fill}%
  \caption{Plots of (a) the cumulative number of surfacings up to time $t$, and (b) evolution of the objective function $V(x(t))$ for our algorithm with $\sigma = 0.5$ and $f(x)=x$ (solid yellow), the algorithm presented in~\cite{bowman_cdc16} with $\sigma = 0.5$ (dot-dash blue), and for a periodic triggering rule with period $T = 0.35$ (dashed red).}
\end{figure*}
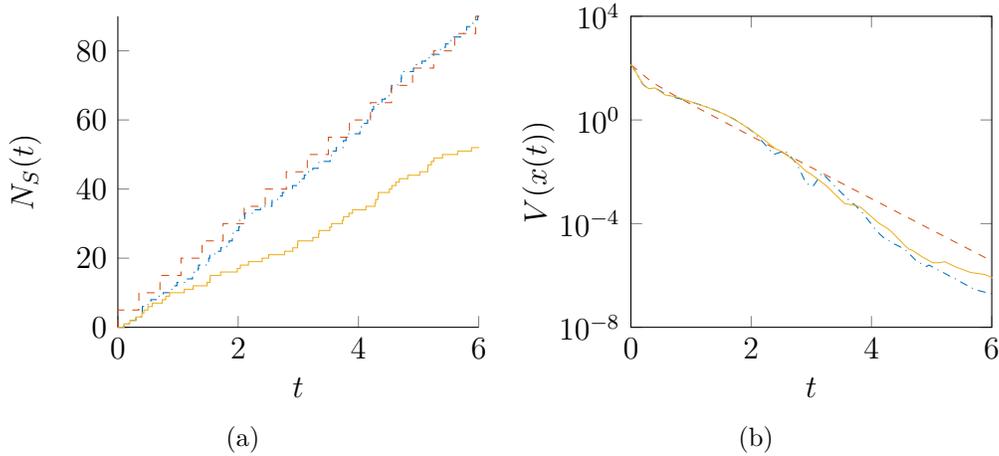

\begin{figure*}[htp]
  \centering
  \hspace*{\fill}%
  \subfigure[]{
%
%
\definecolor{mycolor1}{rgb}{0.00000,0.44700,0.74100}%
\definecolor{mycolor2}{rgb}{0.85000,0.32500,0.09800}%
\definecolor{mycolor3}{rgb}{0.92900,0.69400,0.12500}%
\definecolor{mycolor4}{rgb}{0.49400,0.18400,0.55600}%
\definecolor{mycolor5}{rgb}{0.46600,0.67400,0.18800}%
\begin{tikzpicture}

\begin{axis}[%
width=\figfactor\columnwidth,
scale only axis,
xmin=0,
xmax=6,
xlabel={$t$},
ymin=-2,
ymax=10,
ylabel={$x(t)$},
axis background/.style={fill=white},
axis x line*=bottom,
axis y line*=left
]
\addplot [color=mycolor1,solid,forget plot]
  table[row sep=crcr]{%
0	9\\
0.1	7.85\\
0.2	7.025\\
0.3	6.2\\
0.4	5.375\\
0.43193359375	5.1115478515625\\
0.5	4.55\\
0.5716796875	3.958642578125\\
0.73984375	3.83350165942925\\
0.795703125	3.79193336355586\\
0.86025390625	3.74389727339449\\
1.104296875	3.67128926800178\\
1.2525390625	3.62718404503814\\
1.4833984375	3.55849844089714\\
1.5330078125	3.54373859026109\\
1.53671875	3.5426345069458\\
1.70888671875	3.49141085208091\\
1.98603515625	3.40895326132281\\
2.03642578125	3.39396097209406\\
2.17138671875	3.35380720520622\\
2.398046875	3.28637095850095\\
2.50751953125	3.25380050069974\\
2.77177734375	3.17517814672108\\
2.9091796875	3.13429800923107\\
2.9884765625	3.11070549207265\\
2.99150390625	3.10980479252597\\
3.2416015625	3.03539538803988\\
3.33798828125	3.03977620806952\\
3.34423828125	3.04006027340072\\
3.52353515625	3.04820939758958\\
3.55400390625	3.04959421607919\\
3.73408203125	3.05777884843445\\
3.75185546875	3.05858665922005\\
3.84208984375	3.05654985720675\\
3.9029296875	3.0551765588796\\
4.1248046875	3.05016831843131\\
4.2736328125	3.04680891770807\\
4.28662109375	3.0465157416607\\
4.32431640625	3.04566486982398\\
4.33046875	3.04552599695944\\
4.5140625	3.04580620491514\\
4.554296875	3.04576026776288\\
4.6212890625	3.04568378017441\\
4.67998046875	3.04561676991105\\
4.8205078125	3.04545632437198\\
5.01943359375	3.04522920306337\\
5.15146484375	3.04507845784529\\
5.15556640625	3.04507377493171\\
5.212890625	3.04500832563954\\
5.25302734375	3.04496249998523\\
5.39833984375	3.04479659104698\\
5.64921875	3.04468526941826\\
5.88037109375	3.04458270098884\\
6.02265625	3.04451956528893\\
};
\addplot [color=mycolor2,solid,forget plot]
  table[row sep=crcr]{%
0	-2\\
0.1	0.4\\
0.2	2.8\\
0.3	3.2925\\
0.4	3.785\\
0.43193359375	3.94227294921875\\
0.5	3.74426913609505\\
0.5716796875	3.53575436015129\\
0.73984375	3.04656846890449\\
0.795703125	2.884074665308\\
0.86025390625	2.9389322073343\\
1.104296875	3.14632857272572\\
1.2525390625	3.06796593386443\\
1.4833984375	2.94593083619111\\
1.5330078125	2.9197067119872\\
1.53671875	2.92068458160779\\
1.70888671875	2.96605258532087\\
1.98603515625	3.03908400593218\\
2.03642578125	3.05236244604333\\
2.17138671875	3.05460722750902\\
2.398046875	3.058377225904\\
2.50751953125	3.06019806528392\\
2.77177734375	3.06459341886291\\
2.9091796875	3.06687880780808\\
2.9884765625	3.06819773874161\\
2.99150390625	3.06824809201617\\
3.2416015625	3.06408826181817\\
3.33798828125	3.06248507852827\\
3.34423828125	3.0623811233808\\
3.52353515625	3.05939891008774\\
3.55400390625	3.05889212874382\\
3.73408203125	3.05060571915677\\
3.75185546875	3.04978786311293\\
3.84208984375	3.04563567089036\\
3.9029296875	3.04656048990737\\
4.1248046875	3.0499331846211\\
4.2736328125	3.04824677776683\\
4.28662109375	3.04809960446525\\
4.32431640625	3.04767246992079\\
4.33046875	3.04760275625162\\
4.5140625	3.04552241183822\\
4.554296875	3.04506650657315\\
4.6212890625	3.04430740217549\\
4.67998046875	3.04364235590291\\
4.8205078125	3.0439330823554\\
5.01943359375	3.04434462494311\\
5.15146484375	3.04461777446622\\
5.15556640625	3.0446262598804\\
5.212890625	3.04456231217615\\
5.25302734375	3.04451753788918\\
5.39833984375	3.04435543536119\\
5.64921875	3.04407556883268\\
5.88037109375	3.04416965831959\\
6.02265625	3.04422757483096\\
};
\addplot [color=mycolor3,solid,forget plot]
  table[row sep=crcr]{%
0	0.5\\
0.1	0.6\\
0.2	0.7\\
0.3	0.8\\
0.4	0.9\\
0.43193359375	0.93193359375\\
0.5	1\\
0.5716796875	1.0716796875\\
0.73984375	1.23984375\\
0.795703125	1.295703125\\
0.86025390625	1.36025390625\\
1.104296875	1.604296875\\
1.2525390625	1.7525390625\\
1.4833984375	1.9833984375\\
1.5330078125	2.0330078125\\
1.53671875	2.03671875\\
1.70888671875	2.20888671875\\
1.98603515625	2.48603515625\\
2.03642578125	2.5405063403502\\
2.17138671875	2.68639621714568\\
2.398046875	2.93141098128629\\
2.50751953125	3.04974857310086\\
2.77177734375	3.01290001810151\\
2.9091796875	2.99374040358411\\
2.9884765625	2.98268311361387\\
2.99150390625	2.98297058947504\\
3.2416015625	3.00671980497366\\
3.33798828125	3.01587266545673\\
3.34423828125	3.01646616400882\\
3.52353515625	3.03349215372201\\
3.55400390625	3.03638545916347\\
3.73408203125	3.05348563619566\\
3.75185546875	3.05302472496063\\
3.84208984375	3.05068471407509\\
3.9029296875	3.04910697946286\\
4.1248046875	3.04335318646291\\
4.2736328125	3.03949368798935\\
4.28662109375	3.03960258261563\\
4.32431640625	3.03991862265883\\
4.33046875	3.03997020432392\\
4.5140625	3.04150946671048\\
4.554296875	3.04184679442498\\
4.6212890625	3.04240846144476\\
4.67998046875	3.04245546078301\\
4.8205078125	3.04256799330838\\
5.01943359375	3.04272729056631\\
5.15146484375	3.0428330195269\\
5.15556640625	3.04283630400645\\
5.212890625	3.04288220851819\\
5.25302734375	3.04291434949659\\
5.39833984375	3.04303071391476\\
5.64921875	3.04323161458026\\
5.88037109375	3.04341671846319\\
6.02265625	3.04353065862265\\
};
\addplot [color=mycolor4,solid,forget plot]
  table[row sep=crcr]{%
0	8.5\\
0.1	7.05\\
0.2	5.6\\
0.3	4.15\\
0.4	2.7\\
0.43193359375	2.874357421875\\
0.5	3.246\\
0.5716796875	3.63737109375\\
0.73984375	4.555546875\\
0.795703125	4.41442295419013\\
0.86025390625	4.25134094080669\\
1.104296875	3.63478730775796\\
1.2525390625	3.26026613330099\\
1.4833984375	3.41305175722727\\
1.5330078125	3.44588386253799\\
1.53671875	3.44833980742344\\
1.70888671875	3.37630144369369\\
1.98603515625	3.26033724842142\\
2.03642578125	3.239252849281\\
2.17138671875	3.1827826174747\\
2.398046875	3.08794368258149\\
2.50751953125	3.08329892779077\\
2.77177734375	3.07208687920587\\
2.9091796875	3.06625711115009\\
2.9884765625	3.06289266789401\\
2.99150390625	3.06276422240024\\
3.2416015625	3.05215296725393\\
3.33798828125	3.04806342846853\\
3.34423828125	3.0479488116582\\
3.52353515625	3.04466074191184\\
3.55400390625	3.04468107423644\\
3.73408203125	3.04480124348825\\
3.75185546875	3.04481310401093\\
3.84208984375	3.04487331897225\\
3.9029296875	3.04491391845374\\
4.1248046875	3.0449468113131\\
4.2736328125	3.04496887500925\\
4.28662109375	3.04497080051555\\
4.32431640625	3.04497638882704\\
4.33046875	3.04497547674511\\
4.5140625	3.04494825906219\\
4.554296875	3.04494229433594\\
4.6212890625	3.04493236277717\\
4.67998046875	3.04492366180513\\
4.8205078125	3.0449028286957\\
5.01943359375	3.04487333804671\\
5.15146484375	3.04485376447899\\
5.15556640625	3.04485315642438\\
5.212890625	3.04484465813721\\
5.25302734375	3.04483870788844\\
5.39833984375	3.044794020194\\
5.64921875	3.04471686785057\\
5.88037109375	3.04464578198179\\
6.02265625	3.04460202528098\\
};
\addplot [color=mycolor5,solid,forget plot]
  table[row sep=crcr]{%
0	4\\
0.1	4.1\\
0.2	4.2\\
0.3	4.3\\
0.4	4.4\\
0.43193359375	4.43193359375\\
0.5	4.5\\
0.5716796875	4.4283203125\\
0.73984375	4.26015625\\
0.795703125	4.204296875\\
0.86025390625	4.13974609375\\
1.104296875	3.895703125\\
1.2525390625	3.7474609375\\
1.4833984375	3.5166015625\\
1.5330078125	3.4669921875\\
1.53671875	3.46328125\\
1.70888671875	3.29111328125\\
1.98603515625	3.01396484375\\
2.03642578125	2.96357421875\\
2.17138671875	2.82861328125\\
2.398046875	2.87665441636635\\
2.50751953125	2.89985739459359\\
2.77177734375	2.95586743837853\\
2.9091796875	2.98499017733101\\
2.9884765625	2.99212825682111\\
2.99150390625	2.99240076970804\\
3.2416015625	3.01491385046438\\
3.33798828125	3.02359030915493\\
3.34423828125	3.02415291640538\\
3.52353515625	3.02704123537844\\
3.55400390625	3.02753206082484\\
3.73408203125	3.03043296506575\\
3.75185546875	3.03071927990948\\
3.84208984375	3.03217287834691\\
3.9029296875	3.03315295608123\\
4.1248046875	3.03672717215246\\
4.2736328125	3.03912466567911\\
4.28662109375	3.03933389575723\\
4.32431640625	3.03955650887872\\
4.33046875	3.03959284210839\\
4.5140625	3.04067707181927\\
4.554296875	3.04091467960698\\
4.6212890625	3.04131030810786\\
4.67998046875	3.04165691558459\\
4.8205078125	3.04248681268776\\
5.01943359375	3.04366158711386\\
5.15146484375	3.04444130975701\\
5.15556640625	3.04446553191012\\
5.212890625	3.04480406533581\\
5.25302734375	3.04472855757109\\
5.39833984375	3.04445518639374\\
5.64921875	3.04398321693498\\
5.88037109375	3.04354835834842\\
6.02265625	3.04368577503993\\
};
\end{axis}
\end{tikzpicture}%
\label{fig:states}}
  \hfill
  \subfigure[]{
%
%
\definecolor{mycolor1}{rgb}{0.00000,0.44700,0.74100}%
\definecolor{mycolor2}{rgb}{0.85000,0.32500,0.09800}%
\definecolor{mycolor3}{rgb}{0.92900,0.69400,0.12500}%
\definecolor{mycolor4}{rgb}{0.49400,0.18400,0.55600}%
\definecolor{mycolor5}{rgb}{0.46600,0.67400,0.18800}%
\begin{tikzpicture}

\begin{axis}[%
width=\figfactor\columnwidth,
scale only axis,
xmin=0,
xmax=6,
xlabel={$t$},
ymin=1,
ymax=5,
ylabel={agent id},
axis background/.style={fill=white},
axis x line*=bottom,
axis y line*=left
]
\addplot [color=mycolor1,only marks,mark=x,mark options={solid},forget plot]
  table[row sep=crcr]{%
0.1	1\\
0.5716796875	1\\
0.86025390625	1\\
1.70888671875	1\\
3.2416015625	1\\
3.75185546875	1\\
4.33046875	1\\
4.5140625	1\\
5.01943359375	1\\
5.39833984375	1\\
};
\addplot [color=mycolor2,only marks,mark=x,mark options={solid},forget plot]
  table[row sep=crcr]{%
0.2	2\\
0.43193359375	2\\
0.795703125	2\\
1.104296875	2\\
1.5330078125	2\\
2.03642578125	2\\
2.77177734375	2\\
2.99150390625	2\\
3.55400390625	2\\
3.84208984375	2\\
4.1248046875	2\\
4.67998046875	2\\
5.15556640625	2\\
5.64921875	2\\
};
\addplot [color=mycolor3,only marks,mark=x,mark options={solid},forget plot]
  table[row sep=crcr]{%
0.3	3\\
1.98603515625	3\\
2.50751953125	3\\
2.9884765625	3\\
3.73408203125	3\\
4.2736328125	3\\
4.6212890625	3\\
4.8205078125	3\\
6.02265625	3\\
};
\addplot [color=mycolor4,only marks,mark=x,mark options={solid},forget plot]
  table[row sep=crcr]{%
0.4	4\\
0.73984375	4\\
1.2525390625	4\\
1.53671875	4\\
2.398046875	4\\
3.33798828125	4\\
3.52353515625	4\\
3.9029296875	4\\
4.32431640625	4\\
4.554296875	4\\
5.15146484375	4\\
5.25302734375	4\\
};
\addplot [color=mycolor5,only marks,mark=x,mark options={solid},forget plot]
  table[row sep=crcr]{%
0.5	5\\
1.4833984375	5\\
2.17138671875	5\\
2.9091796875	5\\
3.34423828125	5\\
4.28662109375	5\\
5.212890625	5\\
5.88037109375	5\\
};
\end{axis}
\end{tikzpicture}%
 \label{fig:surface_times}}
  \hspace*{\fill}%
  \caption{Plots of (a) the evolution of the system states $x_i(t)$, $i=1,\dots,5$, and (b) each agent's surfacing times under our algorithm with $\sigma = 0.75$ and $f(x)=x$. }
\end{figure*}
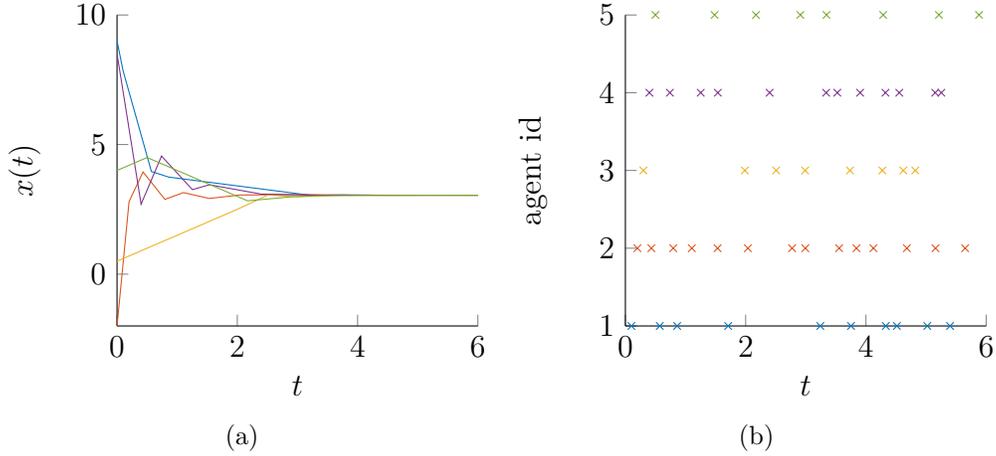

\begin{figure*}[htp]
  \hspace*{\fill}%
  \subfigure[]{
%
%
\definecolor{mycolor1}{rgb}{0.00000,0.44700,0.74100}%
\definecolor{mycolor2}{rgb}{0.85000,0.32500,0.09800}%
\definecolor{mycolor3}{rgb}{0.92900,0.69400,0.12500}%
\begin{tikzpicture}

\begin{axis}[%
width=\figfactor\columnwidth,
scale only axis,
xmin=0,
xmax=6,
xlabel={$t$},
ymin=0,
ymax=100,
ylabel={$N_S(t)$},
axis background/.style={fill=white},
axis x line*=bottom,
axis y line*=left
]
\addplot[const plot,color=mycolor1,dashdotted,forget plot] plot table[row sep=crcr] {%
0	0\\
0.1	1\\
0.2	2\\
0.3	3\\
0.397365557573722	4\\
0.398562877990676	5\\
0.4	6\\
0.5	7\\
0.51120110799191	8\\
0.628753097840896	9\\
0.714834999787085	10\\
0.753070335232496	11\\
0.814834999787085	12\\
0.890920809912167	13\\
0.914834999787085	14\\
1.06405375452911	15\\
1.1367292628778	16\\
1.14530126178643	17\\
1.2367292628778	18\\
1.30444843724274	19\\
1.36461662253582	20\\
1.42075772361799	21\\
1.49994635415543	22\\
1.61113185497434	23\\
1.81646965119184	24\\
1.81684099893146	25\\
1.83003287106383	26\\
1.8392370474141	27\\
1.9403099984792	28\\
1.94083010097689	29\\
2.0403099984792	30\\
2.09595359061775	31\\
2.1403099984792	32\\
2.23003755376857	33\\
2.27311996675251	34\\
2.33003755376857	35\\
2.38928581926495	36\\
2.44692391702562	37\\
2.54692391702562	38\\
2.55170585184456	39\\
2.65170585184456	40\\
2.7990491040591	41\\
2.86342831608764	42\\
2.9308948180542	43\\
2.96677308582246	44\\
3.0308948180542	45\\
3.04069132341142	46\\
3.07763000432352	47\\
3.08278722431564	48\\
3.1308948180542	49\\
3.24000105986011	50\\
3.28419727049238	51\\
3.31179474713674	52\\
3.41179474713674	53\\
3.42652925855598	54\\
3.4919689405041	55\\
3.52652925855598	56\\
3.57649772513024	57\\
3.61414850399778	58\\
3.71579601398467	59\\
3.83069692189202	60\\
3.84189352426997	61\\
3.84630679592151	62\\
3.87890690869046	63\\
3.94189352426997	64\\
3.98786197524713	65\\
4.10444910360161	66\\
4.18275772840566	67\\
4.20444910360161	68\\
4.25099974442038	69\\
4.35099974442038	70\\
4.42063540919012	71\\
4.46487674830694	72\\
4.56487674830694	73\\
4.57241709449123	74\\
4.65495264524978	75\\
4.67241709449123	76\\
4.75495264524978	77\\
4.79359146707993	78\\
4.80617838680441	79\\
4.8287716727378	80\\
4.85495264524978	81\\
4.85495264524978	82\\
4.97050977510402	83\\
5.00620657725344	84\\
5.01396582143486	85\\
5.13415755110046	86\\
5.19316677547189	87\\
5.2192519172445	88\\
5.2745088630973	89\\
5.30491216721419	90\\
5.37576525091615	91\\
5.50695602245402	92\\
5.60695602245402	93\\
5.67050257720903	94\\
5.80281574449857	95\\
5.85572118660139	96\\
5.87579282547049	97\\
6.02134361658221	98\\
};
\addplot[const plot,color=mycolor2,dashed,forget plot] plot table[row sep=crcr] {%
0	0\\
1e-08	1\\
2e-08	2\\
3e-08	3\\
4e-08	4\\
5e-08	5\\
0.35000001	6\\
0.35000002	7\\
0.35000003	8\\
0.35000004	9\\
0.35000005	10\\
0.70000001	11\\
0.70000002	12\\
0.70000003	13\\
0.70000004	14\\
0.70000005	15\\
1.05000001	16\\
1.05000002	17\\
1.05000003	18\\
1.05000004	19\\
1.05000005	20\\
1.40000001	21\\
1.40000002	22\\
1.40000003	23\\
1.40000004	24\\
1.40000005	25\\
1.75000001	26\\
1.75000002	27\\
1.75000003	28\\
1.75000004	29\\
1.75000005	30\\
2.10000001	31\\
2.10000002	32\\
2.10000003	33\\
2.10000004	34\\
2.10000005	35\\
2.45000001	36\\
2.45000002	37\\
2.45000003	38\\
2.45000004	39\\
2.45000005	40\\
2.80000001	41\\
2.80000002	42\\
2.80000003	43\\
2.80000004	44\\
2.80000005	45\\
3.15000001	46\\
3.15000002	47\\
3.15000003	48\\
3.15000004	49\\
3.15000005	50\\
3.50000001	51\\
3.50000002	52\\
3.50000003	53\\
3.50000004	54\\
3.50000005	55\\
3.85000001	56\\
3.85000002	57\\
3.85000003	58\\
3.85000004	59\\
3.85000005	60\\
4.20000001	61\\
4.20000002	62\\
4.20000003	63\\
4.20000004	64\\
4.20000005	65\\
4.55000001	66\\
4.55000002	67\\
4.55000003	68\\
4.55000004	69\\
4.55000005	70\\
4.90000001	71\\
4.90000002	72\\
4.90000003	73\\
4.90000004	74\\
4.90000005	75\\
5.25000001	76\\
5.25000002	77\\
5.25000003	78\\
5.25000004	79\\
5.25000005	80\\
5.60000001	81\\
5.60000002	82\\
5.60000003	83\\
5.60000004	84\\
5.60000005	85\\
5.95000001	86\\
5.95000002	87\\
5.95000003	88\\
5.95000004	89\\
5.95000005	90\\
6.30000001	91\\
};
\addplot[const plot,color=mycolor3,solid,forget plot] plot table[row sep=crcr] {%
0	0\\
0.1	1\\
0.2	2\\
0.3	3\\
0.4	4\\
0.43193359375	5\\
0.5	6\\
0.5716796875	7\\
0.632421875	8\\
0.69755859375	9\\
0.703515625	10\\
0.7876953125	11\\
0.86162109375	12\\
0.9716796875	13\\
1.02216796875	14\\
1.0630859375	15\\
1.06962890625	16\\
1.13271484375	17\\
1.16962890625	18\\
1.22314453125	19\\
1.30380859375	20\\
1.5216796875	21\\
1.541015625	22\\
1.598046875	23\\
1.6837890625	24\\
1.7072265625	25\\
1.75546875	26\\
1.8173828125	27\\
1.90634765625	28\\
1.92646484375	29\\
1.92861328125	30\\
1.94697265625	31\\
2.00634765625	32\\
2.02646484375	33\\
2.08642578125	34\\
2.14013671875	35\\
2.2697265625	36\\
2.33251953125	37\\
2.45712890625	38\\
2.50732421875	39\\
2.568359375	40\\
2.60224609375	41\\
2.7568359375	42\\
2.7900390625	43\\
2.85927734375	44\\
2.96298828125	45\\
3.01025390625	46\\
3.10498046875	47\\
3.12783203125	48\\
3.14619140625	49\\
3.218359375	50\\
3.24560546875	51\\
3.271484375	52\\
3.318359375	53\\
3.36298828125	54\\
3.412109375	55\\
3.4291015625	56\\
3.4697265625	57\\
3.49853515625	58\\
3.69453125	59\\
3.77587890625	60\\
3.9609375	61\\
4.06435546875	62\\
4.08349609375	63\\
4.16435546875	64\\
4.173046875	65\\
4.2177734375	66\\
4.29541015625	67\\
4.39208984375	68\\
4.39873046875	69\\
4.44892578125	70\\
4.56298828125	71\\
4.5634765625	72\\
4.60302734375	73\\
4.66298828125	74\\
4.7150390625	75\\
4.72265625	76\\
4.7884765625	77\\
4.83037109375	78\\
4.9193359375	79\\
4.927734375	80\\
4.95751953125	81\\
5.048828125	82\\
5.18349609375	83\\
5.22822265625	84\\
5.30595703125	85\\
5.3373046875	86\\
5.39453125	87\\
5.4619140625	88\\
5.583984375	89\\
5.62490234375	90\\
5.63994140625	91\\
5.67109375	92\\
5.72490234375	93\\
5.7255859375	94\\
5.8583984375	95\\
5.92412109375	96\\
6.01416015625	97\\
};
\end{axis}
\end{tikzpicture}%
 \label{fig:surfacings2}}
  \hfill
  \subfigure[]{\input{fig/all_compare_V_2.tex}\label{fig:lyap2}}
  \hspace*{\fill}%
  \caption{Plots of (a) the cumulative number of surfacings up to time $t$, and (b) the objective function $V(x(t))$ for our algorithm with $\sigma = 0.75$ and $f(x)=4x$ (solid yellow), the algorithm presented in~\cite{bowman_cdc16} with $\sigma = 0.5$ (dot-dash blue), and for a periodic triggering rule with period $T = 0.35$ (dashed red).}
\end{figure*}

\begin{figure*}[htp]
  \centering
  \hspace*{\fill}%
  \subfigure[]{
%
%
\definecolor{mycolor2}{rgb}{0.00000,0.44700,0.74100}%
\definecolor{mycolor1}{rgb}{0.85000,0.32500,0.09800}%
\begin{tikzpicture}

\begin{axis}[%
width=\figfactor\columnwidth,
scale only axis,
xmin=0,
xmax=6,
xlabel={$t$},
ymin=0,
ymax=70,
ylabel={$N_s(t)$},
axis background/.style={fill=white},
axis x line*=bottom,
axis y line*=left
]
\addplot[const plot,color=mycolor1,dashed,forget plot] plot table[row sep=crcr] {%
0	0\\
1e-08	1\\
2e-08	2\\
3e-08	3\\
4e-08	4\\
5e-08	5\\
0.43000001	6\\
0.43000002	7\\
0.43000003	8\\
0.43000004	9\\
0.43000005	10\\
0.86000001	11\\
0.86000002	12\\
0.86000003	13\\
0.86000004	14\\
0.86000005	15\\
1.29000001	16\\
1.29000002	17\\
1.29000003	18\\
1.29000004	19\\
1.29000005	20\\
1.72000001	21\\
1.72000002	22\\
1.72000003	23\\
1.72000004	24\\
1.72000005	25\\
2.15000001	26\\
2.15000002	27\\
2.15000003	28\\
2.15000004	29\\
2.15000005	30\\
2.58000001	31\\
2.58000002	32\\
2.58000003	33\\
2.58000004	34\\
2.58000005	35\\
3.01000001	36\\
3.01000002	37\\
3.01000003	38\\
3.01000004	39\\
3.01000005	40\\
3.44000001	41\\
3.44000002	42\\
3.44000003	43\\
3.44000004	44\\
3.44000005	45\\
3.87000001	46\\
3.87000002	47\\
3.87000003	48\\
3.87000004	49\\
3.87000005	50\\
4.30000001	51\\
4.30000002	52\\
4.30000003	53\\
4.30000004	54\\
4.30000005	55\\
4.73000001	56\\
4.73000002	57\\
4.73000003	58\\
4.73000004	59\\
4.73000005	60\\
5.16000001	61\\
5.16000002	62\\
5.16000003	63\\
5.16000004	64\\
5.16000005	65\\
5.59000001	66\\
5.59000002	67\\
5.59000003	68\\
5.59000004	69\\
5.59000005	70\\
6.02000001	71\\
};
\addplot[const plot,color=mycolor2,solid,forget plot] plot table[row sep=crcr] {%
0	0\\
0.1	1\\
0.2	2\\
0.3	3\\
0.4	4\\
0.43193359375	5\\
0.5	6\\
0.5716796875	7\\
0.73984375	8\\
0.795703125	9\\
0.86025390625	10\\
1.104296875	11\\
1.2525390625	12\\
1.4833984375	13\\
1.5330078125	14\\
1.53671875	15\\
1.70888671875	16\\
1.98603515625	17\\
2.03642578125	18\\
2.17138671875	19\\
2.398046875	20\\
2.50751953125	21\\
2.77177734375	22\\
2.9091796875	23\\
2.9884765625	24\\
2.99150390625	25\\
3.2416015625	26\\
3.33798828125	27\\
3.34423828125	28\\
3.52353515625	29\\
3.55400390625	30\\
3.73408203125	31\\
3.75185546875	32\\
3.84208984375	33\\
3.9029296875	34\\
4.1248046875	35\\
4.2736328125	36\\
4.28662109375	37\\
4.32431640625	38\\
4.33046875	39\\
4.5140625	40\\
4.554296875	41\\
4.6212890625	42\\
4.67998046875	43\\
4.8205078125	44\\
5.01943359375	45\\
5.15146484375	46\\
5.15556640625	47\\
5.212890625	48\\
5.25302734375	49\\
5.39833984375	50\\
5.64921875	51\\
5.88037109375	52\\
6.02265625	53\\
};
\end{axis}
\end{tikzpicture}%
 \label{fig:surfacings_thresh}}
  \hfill
  \subfigure[]{
%
%
\definecolor{mycolor2}{rgb}{0.00000,0.44700,0.74100}%
\definecolor{mycolor1}{rgb}{0.85000,0.32500,0.09800}%
\begin{tikzpicture}

\begin{axis}[%
width=\figfactor\columnwidth,
scale only axis,
xmin=0,
xmax=6,
xlabel={$t$},
ymode=log,
ymin=1e-08,
ymax=10000,
yminorticks=true,
ylabel={$V(x(t))$},
axis background/.style={fill=white}
]
\addplot [color=mycolor1,dashed,forget plot]
  table[row sep=crcr]{%
0	135.125\\
1e-08	135.124990795\\
2e-08	135.124981590001\\
3e-08	135.124972385002\\
4e-08	135.124963180004\\
5e-08	135.124953975006\\
0.43000001	64.803313381938\\
0.43000002	64.80331897028\\
0.43000003	64.8033171717828\\
0.43000004	64.8033142962652\\
0.43000005	64.8033089672977\\
0.86000001	57.05018698623\\
0.86000002	57.0501921709051\\
0.86000003	57.0501923594425\\
0.86000004	57.0501908923909\\
0.86000005	57.0501868214644\\
1.29000001	53.9030109269542\\
1.29000002	53.9030159473007\\
1.29000003	53.9030166814984\\
1.29000004	53.9030158752081\\
1.29000005	53.9030121200782\\
1.72000001	51.9519579246636\\
1.72000002	51.9519627819538\\
1.72000003	51.9519637850997\\
1.72000004	51.9519633745499\\
1.72000005	51.9519598307497\\
2.15000001	50.3584854323518\\
2.15000002	50.3584901482977\\
2.15000003	50.3584912753514\\
2.15000004	50.3584910600103\\
2.15000005	50.3584876553868\\
2.58000001	48.8969749919772\\
2.58000002	48.8969795725667\\
2.58000003	48.8969807533653\\
2.58000004	48.896980652157\\
2.58000005	48.8969773682083\\
3.01000001	47.5025716270446\\
3.01000002	47.5025760775954\\
3.01000003	47.5025772707206\\
3.01000004	47.5025772281935\\
3.01000005	47.5025740478976\\
3.44000001	46.1553941421919\\
3.44000002	46.1553984666668\\
3.44000003	46.1553996515546\\
3.44000004	46.1553996420947\\
3.44000005	46.1553965583009\\
3.87000001	44.8486997820695\\
3.87000002	44.8487039841661\\
3.87000003	44.8487051493123\\
3.87000004	44.8487051570346\\
3.87000005	44.8487021636745\\
4.30000001	43.5796976630415\\
4.30000002	43.5797017462528\\
4.30000003	43.5797028860658\\
4.30000004	43.5797029030468\\
4.30000005	43.5796999962316\\
4.73000001	42.3468172466613\\
4.73000002	42.3468212143622\\
4.73000003	42.3468223260775\\
4.73000004	42.3468223476834\\
4.73000005	42.3468195240619\\
5.16000001	41.1488816212095\\
5.16000002	41.1488854766706\\
5.16000003	41.1488865592218\\
5.16000004	41.1488865830474\\
5.16000005	41.1488838398491\\
5.59000001	39.9848544723324\\
5.59000002	39.98485821873\\
5.59000003	39.9848592719039\\
5.59000004	39.9848592965926\\
5.59000005	39.9848566312855\\
6.02000001	38.8537618533724\\
};
\addplot [color=mycolor2,solid,forget plot]
  table[row sep=crcr]{%
0	135.125\\
0.1	60.67875\\
0.2	23.170625\\
0.3	15.938209375\\
0.4	17.1620875\\
0.43193359375	15.6254274778775\\
0.5	11.9757097241789\\
0.5716796875	9.12833835828197\\
0.73984375	7.94571319931764\\
0.795703125	7.29033065673603\\
0.86025390625	6.42906996272308\\
1.104296875	4.10599474176634\\
1.2525390625	3.21587646751061\\
1.4833984375	1.95124938310368\\
1.5330078125	1.7594227070401\\
1.53671875	1.74541632093954\\
1.70888671875	1.10466355762484\\
1.98603515625	0.426557564774064\\
2.03642578125	0.346266994543607\\
2.17138671875	0.208219963526346\\
2.398046875	0.0779952819923297\\
2.50751953125	0.0616568462974229\\
2.77177734375	0.0211724123398144\\
2.9091796875	0.0106027031015892\\
2.9884765625	0.0082653413706619\\
2.99150390625	0.00814102181731084\\
3.2416015625	0.00299579138229799\\
3.33798828125	0.00181178372772009\\
3.34423828125	0.00175112764437313\\
3.52353515625	0.000689116565798661\\
3.55400390625	0.000595782855815324\\
3.73408203125	0.000499865472380671\\
3.75185546875	0.000499261947752058\\
3.84208984375	0.000392762139859344\\
3.9029296875	0.000290802668819166\\
4.1248046875	0.000103473126296758\\
4.2736328125	6.35526929927497e-05\\
4.28662109375	5.9365573436115e-05\\
4.32431640625	5.07008449515418e-05\\
4.33046875	4.94448030356337e-05\\
4.5140625	1.80929556622585e-05\\
4.554296875	1.43114392736569e-05\\
4.6212890625	1.03904084680509e-05\\
4.67998046875	9.36923983506019e-06\\
4.8205078125	5.63711171036336e-06\\
5.01943359375	3.07283838033088e-06\\
5.15146484375	3.13023706169629e-06\\
5.15556640625	3.15450132137591e-06\\
5.212890625	3.41168200757589e-06\\
5.25302734375	3.0950814833273e-06\\
5.39833984375	2.14289935175768e-06\\
5.64921875	1.29970312120356e-06\\
5.88037109375	1.09493165754765e-06\\
6.02265625	7.90769515650281e-07\\
};
\end{axis}
\end{tikzpicture}%
 \label{fig:lyap_thresh}}
  \hspace*{\fill}%
  \caption{Plots of (a) the cumulative number of surfacings up to time $t$, and (b) the objective function $V(x(t))$ for our algorithm with $\sigma = 0.75$ and $f(x) = x$, and a periodic triggering rule near the threshold of its convergence, $T = 0.43$. Our algorithm's surfacings are shown as solid lines in blue, the periodic one as dashed red. }
\end{figure*}
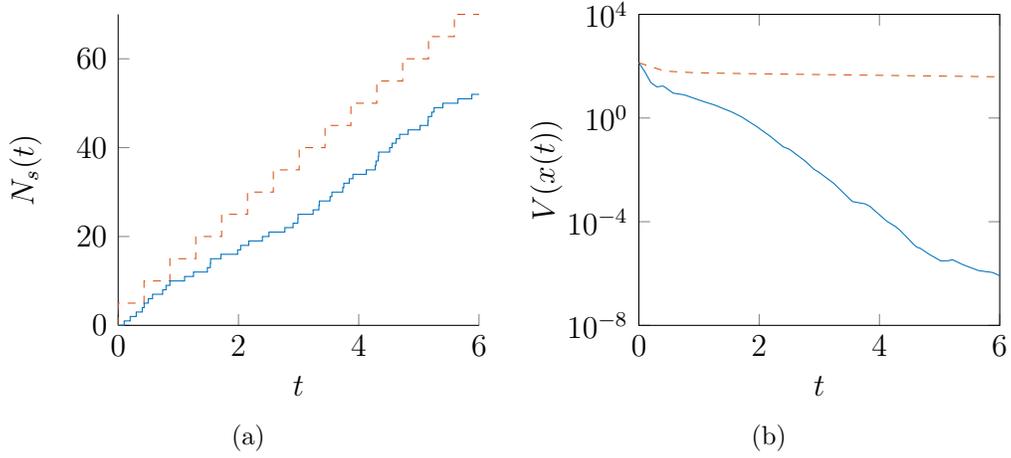

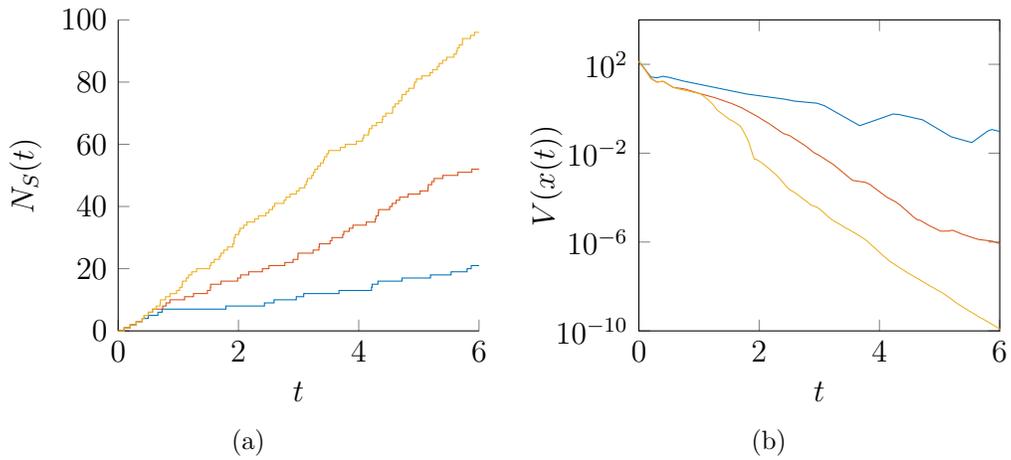
\begin{figure*}[htp]
  \centering
  \hspace*{\fill}%
  \subfigure[]{
%
%
\definecolor{mycolor1}{rgb}{0.00000,0.44700,0.74100}%
\definecolor{mycolor2}{rgb}{0.85000,0.32500,0.09800}%
\definecolor{mycolor3}{rgb}{0.92900,0.69400,0.12500}%
\begin{tikzpicture}

\begin{axis}[%
width=\figfactor\columnwidth,
scale only axis,
xmin=0,
xmax=6,
xlabel={$t$},
ymin=0,
ymax=100,
ylabel={$N_S(t)$},
axis background/.style={fill=white},
axis x line*=bottom,
axis y line*=left
]
\addplot[const plot,color=mycolor1,solid,forget plot] plot table[row sep=crcr] {%
0	0\\
0.1	1\\
0.2	2\\
0.3	3\\
0.4	4\\
0.5	5\\
0.6609375	6\\
0.73037109375	7\\
1.790625	8\\
2.43212890625	9\\
2.5916015625	10\\
2.961328125	11\\
3.08515625	12\\
3.6716796875	13\\
4.21259765625	14\\
4.2212890625	15\\
4.31884765625	16\\
4.719140625	17\\
5.19306640625	18\\
5.5359375	19\\
5.80126953125	20\\
5.87021484375	21\\
6.277734375	22\\
};
\addplot[const plot,color=mycolor2,solid,forget plot] plot table[row sep=crcr] {%
0	0\\
0.1	1\\
0.2	2\\
0.3	3\\
0.4	4\\
0.43193359375	5\\
0.5	6\\
0.5716796875	7\\
0.73984375	8\\
0.795703125	9\\
0.86025390625	10\\
1.104296875	11\\
1.2525390625	12\\
1.4833984375	13\\
1.5330078125	14\\
1.53671875	15\\
1.70888671875	16\\
1.98603515625	17\\
2.03642578125	18\\
2.17138671875	19\\
2.398046875	20\\
2.50751953125	21\\
2.77177734375	22\\
2.9091796875	23\\
2.9884765625	24\\
2.99150390625	25\\
3.2416015625	26\\
3.33798828125	27\\
3.34423828125	28\\
3.52353515625	29\\
3.55400390625	30\\
3.73408203125	31\\
3.75185546875	32\\
3.84208984375	33\\
3.9029296875	34\\
4.1248046875	35\\
4.2736328125	36\\
4.28662109375	37\\
4.32431640625	38\\
4.33046875	39\\
4.5140625	40\\
4.554296875	41\\
4.6212890625	42\\
4.67998046875	43\\
4.8205078125	44\\
5.01943359375	45\\
5.15146484375	46\\
5.15556640625	47\\
5.212890625	48\\
5.25302734375	49\\
5.39833984375	50\\
5.64921875	51\\
5.88037109375	52\\
6.02265625	53\\
};
\addplot[const plot,color=mycolor3,solid,forget plot] plot table[row sep=crcr] {%
0	0\\
0.1	1\\
0.2	2\\
0.3	3\\
0.4	4\\
0.43193359375	5\\
0.5	6\\
0.5716796875	7\\
0.632421875	8\\
0.69755859375	9\\
0.703515625	10\\
0.7876953125	11\\
0.86162109375	12\\
0.9716796875	13\\
1.02216796875	14\\
1.0630859375	15\\
1.06962890625	16\\
1.13271484375	17\\
1.16962890625	18\\
1.22314453125	19\\
1.30380859375	20\\
1.5216796875	21\\
1.541015625	22\\
1.598046875	23\\
1.6837890625	24\\
1.7072265625	25\\
1.75546875	26\\
1.8173828125	27\\
1.90634765625	28\\
1.92646484375	29\\
1.92861328125	30\\
1.94697265625	31\\
2.00634765625	32\\
2.02646484375	33\\
2.08642578125	34\\
2.14013671875	35\\
2.2697265625	36\\
2.33251953125	37\\
2.45712890625	38\\
2.50732421875	39\\
2.568359375	40\\
2.60224609375	41\\
2.7568359375	42\\
2.7900390625	43\\
2.85927734375	44\\
2.96298828125	45\\
3.01025390625	46\\
3.10498046875	47\\
3.12783203125	48\\
3.14619140625	49\\
3.218359375	50\\
3.24560546875	51\\
3.271484375	52\\
3.318359375	53\\
3.36298828125	54\\
3.412109375	55\\
3.4291015625	56\\
3.4697265625	57\\
3.49853515625	58\\
3.69453125	59\\
3.77587890625	60\\
3.9609375	61\\
4.06435546875	62\\
4.08349609375	63\\
4.16435546875	64\\
4.173046875	65\\
4.2177734375	66\\
4.29541015625	67\\
4.39208984375	68\\
4.39873046875	69\\
4.44892578125	70\\
4.56298828125	71\\
4.5634765625	72\\
4.60302734375	73\\
4.66298828125	74\\
4.7150390625	75\\
4.72265625	76\\
4.7884765625	77\\
4.83037109375	78\\
4.9193359375	79\\
4.927734375	80\\
4.95751953125	81\\
5.048828125	82\\
5.18349609375	83\\
5.22822265625	84\\
5.30595703125	85\\
5.3373046875	86\\
5.39453125	87\\
5.4619140625	88\\
5.583984375	89\\
5.62490234375	90\\
5.63994140625	91\\
5.67109375	92\\
5.72490234375	93\\
5.7255859375	94\\
5.8583984375	95\\
5.92412109375	96\\
6.01416015625	97\\
};
\end{axis}
\end{tikzpicture}%
 \label{fig:compare_f_Ns}}
  \hfill
  \subfigure[]{
%
%
\definecolor{mycolor1}{rgb}{0.00000,0.44700,0.74100}%
\definecolor{mycolor2}{rgb}{0.85000,0.32500,0.09800}%
\definecolor{mycolor3}{rgb}{0.92900,0.69400,0.12500}%
\begin{tikzpicture}

\begin{axis}[%
width=\figfactor\columnwidth,
scale only axis,
xmin=0,
xmax=6,
xlabel={$t$},
ymode=log,
ymin=1e-10,
ymax=10000,
yminorticks=true,
ylabel={$V(x(t))$},
axis background/.style={fill=white}
]
\addplot [color=mycolor1,solid,forget plot]
  table[row sep=crcr]{%
0	135.125\\
0.1	60.67875\\
0.2	27.0494140625\\
0.3	24.6807010131836\\
0.4	28.9054388183594\\
0.5	25.8845569502869\\
0.6609375	20.0604149708152\\
0.73037109375	17.8895152597421\\
1.790625	4.61335512087811\\
2.43212890625	2.67224360333981\\
2.5916015625	2.20270636965425\\
2.961328125	1.77435459573398\\
3.08515625	1.37028465296194\\
3.6716796875	0.172547247724319\\
4.21259765625	0.560895704729973\\
4.2212890625	0.578205092877316\\
4.31884765625	0.553968190323539\\
4.719140625	0.316210683125134\\
5.19306640625	0.0545252867050041\\
5.5359375	0.0300370422930417\\
5.80126953125	0.0992100783884016\\
5.87021484375	0.116610720208568\\
6.277734375	0.060206188862978\\
};
\addplot [color=mycolor2,solid,forget plot]
  table[row sep=crcr]{%
0	135.125\\
0.1	60.67875\\
0.2	23.170625\\
0.3	15.938209375\\
0.4	17.1620875\\
0.43193359375	15.6254274778775\\
0.5	11.9757097241789\\
0.5716796875	9.12833835828197\\
0.73984375	7.94571319931764\\
0.795703125	7.29033065673603\\
0.86025390625	6.42906996272308\\
1.104296875	4.10599474176634\\
1.2525390625	3.21587646751061\\
1.4833984375	1.95124938310368\\
1.5330078125	1.7594227070401\\
1.53671875	1.74541632093954\\
1.70888671875	1.10466355762484\\
1.98603515625	0.426557564774064\\
2.03642578125	0.346266994543607\\
2.17138671875	0.208219963526346\\
2.398046875	0.0779952819923297\\
2.50751953125	0.0616568462974229\\
2.77177734375	0.0211724123398144\\
2.9091796875	0.0106027031015892\\
2.9884765625	0.0082653413706619\\
2.99150390625	0.00814102181731084\\
3.2416015625	0.00299579138229799\\
3.33798828125	0.00181178372772009\\
3.34423828125	0.00175112764437313\\
3.52353515625	0.000689116565798661\\
3.55400390625	0.000595782855815324\\
3.73408203125	0.000499865472380671\\
3.75185546875	0.000499261947752058\\
3.84208984375	0.000392762139859344\\
3.9029296875	0.000290802668819166\\
4.1248046875	0.000103473126296758\\
4.2736328125	6.35526929927497e-05\\
4.28662109375	5.9365573436115e-05\\
4.32431640625	5.07008449515418e-05\\
4.33046875	4.94448030356337e-05\\
4.5140625	1.80929556622585e-05\\
4.554296875	1.43114392736569e-05\\
4.6212890625	1.03904084680509e-05\\
4.67998046875	9.36923983506019e-06\\
4.8205078125	5.63711171036336e-06\\
5.01943359375	3.07283838033088e-06\\
5.15146484375	3.13023706169629e-06\\
5.15556640625	3.15450132137591e-06\\
5.212890625	3.41168200757589e-06\\
5.25302734375	3.0950814833273e-06\\
5.39833984375	2.14289935175768e-06\\
5.64921875	1.29970312120356e-06\\
5.88037109375	1.09493165754765e-06\\
6.02265625	7.90769515650281e-07\\
};
\addplot [color=mycolor3,solid,forget plot]
  table[row sep=crcr]{%
0	135.125\\
0.1	60.67875\\
0.2	23.170625\\
0.3	15.938209375\\
0.4	17.1620875\\
0.43193359375	15.6254274778775\\
0.5	11.9757097241789\\
0.5716796875	9.12833835828197\\
0.632421875	8.13273730144449\\
0.69755859375	7.3774547775148\\
0.703515625	7.3217727789192\\
0.7876953125	6.56843641634006\\
0.86162109375	5.92864114431566\\
0.9716796875	5.04509430612827\\
1.02216796875	4.66400836115476\\
1.0630859375	3.91202990600568\\
1.06962890625	3.79846595656352\\
1.13271484375	2.79519900105623\\
1.16962890625	2.28722047993873\\
1.22314453125	1.65053379392998\\
1.30380859375	0.862596137998336\\
1.5216796875	0.30932589762499\\
1.541015625	0.301702799434413\\
1.598046875	0.239742713199075\\
1.6837890625	0.170373211209357\\
1.7072265625	0.134268005754771\\
1.75546875	0.0796058956556684\\
1.8173828125	0.0368330067497289\\
1.90634765625	0.00669287338020297\\
1.92646484375	0.00524223849852395\\
1.92861328125	0.00526107375212845\\
1.94697265625	0.00534825396283584\\
2.00634765625	0.004272757404087\\
2.02646484375	0.00393341523507171\\
2.08642578125	0.00297539327611099\\
2.14013671875	0.00228487752443826\\
2.2697265625	0.0011286349749981\\
2.33251953125	0.000768745346848614\\
2.45712890625	0.000306969215417549\\
2.50732421875	0.000227821807100987\\
2.568359375	0.000183423521338809\\
2.60224609375	0.000161288743144472\\
2.7568359375	7.7675148135347e-05\\
2.7900390625	6.73326140692397e-05\\
2.85927734375	4.66713289558347e-05\\
2.96298828125	3.70309666730406e-05\\
3.01025390625	2.97337446202479e-05\\
3.10498046875	1.77148008225769e-05\\
3.12783203125	1.5510706489235e-05\\
3.14619140625	1.38977586323667e-05\\
3.218359375	9.84570919183858e-06\\
3.24560546875	8.8133264688193e-06\\
3.271484375	8.01452672924202e-06\\
3.318359375	6.80791657845263e-06\\
3.36298828125	5.76702278094971e-06\\
3.412109375	4.717293056566e-06\\
3.4291015625	4.40164566754483e-06\\
3.4697265625	3.91992262961183e-06\\
3.49853515625	3.59179207501002e-06\\
3.69453125	1.72011134506834e-06\\
3.77587890625	1.218299545551e-06\\
3.9609375	4.35672899118606e-07\\
4.06435546875	2.41209511828351e-07\\
4.08349609375	2.17579827334422e-07\\
4.16435546875	1.41016439828934e-07\\
4.173046875	1.34094250942387e-07\\
4.2177734375	1.12587968183176e-07\\
4.29541015625	8.2310383579668e-08\\
4.39208984375	5.65561940523633e-08\\
4.39873046875	5.51893996220154e-08\\
4.44892578125	4.46099836199179e-08\\
4.56298828125	2.93416784815921e-08\\
4.5634765625	2.92897506367555e-08\\
4.60302734375	2.51946604584751e-08\\
4.66298828125	2.05607602809458e-08\\
4.7150390625	1.69351452197319e-08\\
4.72265625	1.64810369684585e-08\\
4.7884765625	1.29138627839485e-08\\
4.83037109375	1.09080886742923e-08\\
4.9193359375	8.09861397468327e-09\\
4.927734375	7.84769241813353e-09\\
4.95751953125	7.00463641180749e-09\\
5.048828125	5.10665503617282e-09\\
5.18349609375	2.88104059469193e-09\\
5.22822265625	2.30773831299972e-09\\
5.30595703125	1.61170977318282e-09\\
5.3373046875	1.44186595527892e-09\\
5.39453125	1.1593336690771e-09\\
5.4619140625	8.78223925907548e-10\\
5.583984375	5.63994148448542e-10\\
5.62490234375	4.78404316040112e-10\\
5.63994140625	4.50054304662205e-10\\
5.67109375	3.99927891642232e-10\\
5.72490234375	3.38402337544312e-10\\
5.7255859375	3.37634218275899e-10\\
5.8583984375	2.07961139106515e-10\\
5.92412109375	1.63019887628076e-10\\
6.01416015625	1.1092638655221e-10\\
};
\end{axis}
\end{tikzpicture}%
 \label{fig:compare_f_V}}
  \hspace*{\fill}%
  \caption{Plots of (a) the cumulative number of surfacings up to time $t$, and (b) the objective function $V(x(t))$ for our algorithm with $\sigma = 0.5$ and three promise selection functions: $f(x) = 0.25x$, $f(x) = x$, and $f(x) = 4x$, shown in blue, red, and yellow, respectively. }
\end{figure*}

The evolution of the objective function $V(x(t))$ for the same three configurations described above is displayed in Figure~\ref{fig:lyap}.
Note that although all three algorithms have a similar convergence rate, the algorithm presented here requires significantly fewer communications amongst the agents to achieve that result. 
For our algorithm under these parameters, the evolution of the robot states over time is shown in Figure~\ref{fig:states}, and each individual agent's surfacing times are shown in Figure~\ref{fig:surface_times}.

We also ran the same simulation but with our algorithm having $\sigma = 0.75$ and the promise function as $f(x) = 4x$.
The resulting surfacing counts and objective function evolution can be seen in Figures~\ref{fig:surfacings2} and~\ref{fig:lyap2}. 
In this situation, note that although all algorithms resulted in a similar number of communications required, the algorithm presented here converged more quickly.

As mentioned above, the periodic triggering rule for this specific network topology is guaranteed to converge for any period $T \leq 0.4331$. 
We further compared our algorithm with $f(x) = x$ and $\sigma = 0.75$ against the periodic triggering rule with a period very near this threshold, $T = 0.43$. 
The resulting $N_S(t)$ and evolution of the global objective $V(t)$ is seen in Figures~\ref{fig:surfacings_thresh} and~\ref{fig:lyap_thresh} respectively.
Here our algorithm both converges significantly more quickly and requires far fewer surfacings by the agents than even the most infrequent possible communication under a periodic triggering rule. 
Furthermore, to determine the threshold under which a periodic triggering rule will converge, each agent required global information about the communication graph.
On the contrary, our algorithm is guaranteed to converge using only local information and no shared parameters.

We additionally investigated the effect of choosing various promise functions $f(x)$.
We ran three simulations with $\sigma = 0.5$ and $f(x) = 0.25x$, $f(x) = x$, and $f(x) = 4x$. 
The results can be seen in Figures~\ref{fig:compare_f_Ns} and~\ref{fig:compare_f_V}.
For a promise of the form $f(x) = cx$, we see that a smaller value of $c$ results in more infrequent communication while also slowing convergence; it forces agents to move more slowly, slowing their movement towards consensus, while also slowing the growth of the bound their neighbors can make on their state, reducing the rate at which those neighbors need to communicate.

A single agent's control law (agent 5) from a run of our algorithm with $\sigma = 0.5$ is shown in Figure~\ref{fig:promises}, along with its ``promise'' currently on the cloud server. 
There exists a lag between when the promised control max $M_5(t)$ increases and when the actual control increases likewise.
While $M_5(t)$ represents the ideal control that the agent would use, it is still bound to a previous promise until the newer one propagates to all neighbor agents.

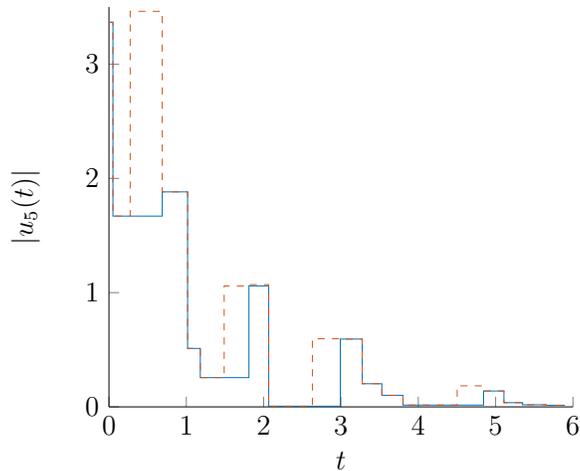
\begin{figure}[t]
  \centering
  \scalebox{.9}{
%
%
\definecolor{mycolor1}{rgb}{0.00000,0.44700,0.74100}%
\definecolor{mycolor2}{rgb}{0.85000,0.32500,0.09800}%
\begin{tikzpicture}

\begin{axis}[%
width=0.5\columnwidth,
scale only axis,
xmin=0,
xmax=6,
xlabel={$t$},
ymin=0,
ymax=3.5,
ylabel={$|u_5(t)|$},
axis background/.style={fill=white},
axis x line*=bottom,
axis y line*=left
]
\addplot[const plot,color=mycolor1,solid] plot table[row sep=crcr] {%
0	3.3666267970356\\
2e-09	3.3666267970356\\
2e-09	3.3666267970356\\
2e-09	3.3666267970356\\
2e-09	3.3666267970356\\
2e-09	3.36662660092804\\
0.037890627	3.36662660092804\\
0.0513671895	1.67012024777672\\
0.119921877	1.67012024777672\\
0.1302734395	1.67012024777672\\
0.170703127	1.67012024777672\\
0.202343752	1.67012024777672\\
0.2251953145	1.67012024777672\\
0.2439453145	1.67012024777672\\
0.250781252	1.67012024777672\\
0.274218752	1.67012024777672\\
0.286328127	1.67012024777672\\
0.298437502	1.67012024777672\\
0.362500002	1.67012024777672\\
0.3716796895	1.67012024777672\\
0.4599609395	1.67012024777672\\
0.482031252	1.67012024777672\\
0.666796877	1.67012024777672\\
0.688671877	1.88189085498726\\
0.8005859395	1.88189085498726\\
0.8634765645	1.88189085498726\\
1.016796877	0.511589360320372\\
1.0970703145	0.511589360320372\\
1.169921877	0.511589360320372\\
1.1810546895	0.255984963563767\\
1.188671877	0.255984963563767\\
1.276953127	0.255984963563767\\
1.316796877	0.255984963563767\\
1.406640627	0.255984963563767\\
1.439062502	0.255984963563767\\
1.441015627	0.255984963563767\\
1.4896484395	0.255984963563767\\
1.5544921895	0.255984963563767\\
1.590625002	0.255984963563767\\
1.6091796895	0.255984963563767\\
1.6310546895	0.255984963563767\\
1.670703127	0.255984963563767\\
1.674218752	0.255984963563767\\
1.8099609395	1.05882397314022\\
1.8416015645	1.05882397314022\\
1.8873046895	1.05882397314022\\
1.992968752	1.05882397314022\\
2.065625002	0.00608450190486542\\
2.139453127	0.00608450190486542\\
2.221093752	0.00608450190486542\\
2.277343752	0.00608450190486542\\
2.344921877	0.00608450190486542\\
2.386718752	0.00608450190486542\\
2.4193359395	0.00608450190486542\\
2.630468752	0.00608450190486542\\
2.6865234395	0.00608450190486542\\
2.7787109395	0.00608450190486542\\
2.8130859395	0.00608450190486542\\
2.827343752	0.00608450190486542\\
2.853125002	0.00608450190486542\\
2.868750002	0.00608450190486542\\
2.965625002	0.00608450190486542\\
2.9763671895	0.00608450190486542\\
2.9916015645	0.593613378018525\\
3.006640627	0.593613378018525\\
3.0134765645	0.593613378018525\\
3.1318359395	0.593613378018525\\
3.181640627	0.593613378018525\\
3.2744140645	0.202493411834104\\
3.2962890645	0.202493411834104\\
3.337109377	0.202493411834104\\
3.369140627	0.202493411834104\\
3.426171877	0.202493411834104\\
3.528906252	0.101265837828103\\
3.6064453145	0.101265837828103\\
3.628515627	0.101265837828103\\
3.658203127	0.101265837828103\\
3.668359377	0.101265837828103\\
3.802343752	0.0159150206489777\\
3.8974609395	0.0159150206489777\\
3.948437502	0.0159150206489777\\
4.076171877	0.0159150206489777\\
4.1841796895	0.0159150206489777\\
4.355859377	0.0159150206489777\\
4.3927734395	0.0159150206489777\\
4.4466796895	0.0159150206489777\\
4.5009765645	0.0159150206489777\\
4.502734377	0.0159150206489777\\
4.509375002	0.0159150206489777\\
4.537109377	0.0159150206489777\\
4.5416015645	0.0159150206489777\\
4.612109377	0.0159150206489777\\
4.634375002	0.0159150206489777\\
4.774609377	0.0159150206489777\\
4.819140627	0.0159150206489777\\
4.8392578145	0.0159150206489777\\
4.844531252	0.138885990904913\\
4.879296877	0.138885990904913\\
4.9740234395	0.138885990904913\\
5.024218752	0.138885990904913\\
5.107421877	0.03674945207232\\
5.170703127	0.03674945207232\\
5.205468752	0.03674945207232\\
5.288671877	0.03674945207232\\
5.295703127	0.03674945207232\\
5.3392578145	0.03674945207232\\
5.348046877	0.0183680759440234\\
5.4181640645	0.0183680759440234\\
5.5216796895	0.0183680759440234\\
5.569531252	0.0183680759440234\\
5.5919921895	0.0183680759440234\\
5.5998046895	0.0183680759440234\\
5.6525390645	0.0142509221326956\\
5.727343752	0.0142509221326956\\
5.7447265645	0.0142509221326956\\
5.790234377	0.0142509221326956\\
5.854687502	0.0142509221326956\\
5.886718752	0.00959451552871876\\
};

\addplot[const plot,color=mycolor2,dashed] plot table[row sep=crcr] {%
0	3.3666267970356\\
2e-09	3.3666267970356\\
2e-09	3.3666267970356\\
2e-09	3.3666267970356\\
2e-09	3.3666267970356\\
2e-09	3.36662660092804\\
0.037890627	3.36662660092804\\
0.0513671895	1.67012024777672\\
0.119921877	1.67012024777672\\
0.1302734395	1.67012024777672\\
0.170703127	1.67012024777672\\
0.202343752	1.67012024777672\\
0.2251953145	1.67012024777672\\
0.2439453145	1.67012024777672\\
0.250781252	1.67012024777672\\
0.274218752	3.46221704341609\\
0.286328127	3.46221704341609\\
0.298437502	3.46221704341609\\
0.362500002	3.46221704341609\\
0.3716796895	3.46221704341609\\
0.4599609395	3.46221704341609\\
0.482031252	3.46221704341609\\
0.666796877	3.46221704341609\\
0.688671877	1.88189085498726\\
0.8005859395	1.88189085498726\\
0.8634765645	1.88189085498726\\
1.016796877	0.511589360320372\\
1.0970703145	0.511589360320372\\
1.169921877	0.511589360320372\\
1.1810546895	0.255984963563767\\
1.188671877	0.255984963563767\\
1.276953127	0.255984963563767\\
1.316796877	0.255984963563767\\
1.406640627	0.255984963563767\\
1.439062502	0.255984963563767\\
1.441015627	0.255984963563767\\
1.4896484395	1.05882397314022\\
1.5544921895	1.05882397314022\\
1.590625002	1.05882397314022\\
1.6091796895	1.05882397314022\\
1.6310546895	1.05882397314022\\
1.670703127	1.05882397314022\\
1.674218752	1.05882397314022\\
1.8099609395	1.07012246196572\\
1.8416015645	1.07012246196572\\
1.8873046895	1.07012246196572\\
1.992968752	1.07012246196572\\
2.065625002	0.00608450190486542\\
2.139453127	0.00608450190486542\\
2.221093752	0.00608450190486542\\
2.277343752	0.00608450190486542\\
2.344921877	0.00608450190486542\\
2.386718752	0.00608450190486542\\
2.4193359395	0.00608450190486542\\
2.630468752	0.596258802844438\\
2.6865234395	0.596258802844438\\
2.7787109395	0.596258802844438\\
2.8130859395	0.596258802844438\\
2.827343752	0.596258802844438\\
2.853125002	0.596258802844438\\
2.868750002	0.596258802844438\\
2.965625002	0.596258802844438\\
2.9763671895	0.596258802844438\\
2.9916015645	0.593613378018525\\
3.006640627	0.593613378018525\\
3.0134765645	0.593613378018525\\
3.1318359395	0.593613378018525\\
3.181640627	0.593613378018525\\
3.2744140645	0.202493411834104\\
3.2962890645	0.202493411834104\\
3.337109377	0.202493411834104\\
3.369140627	0.202493411834104\\
3.426171877	0.202493411834104\\
3.528906252	0.101265837828103\\
3.6064453145	0.101265837828103\\
3.628515627	0.101265837828103\\
3.658203127	0.101265837828103\\
3.668359377	0.101265837828103\\
3.802343752	0.0159150206489777\\
3.8974609395	0.0159150206489777\\
3.948437502	0.0159150206489777\\
4.076171877	0.0159150206489777\\
4.1841796895	0.0159150206489777\\
4.355859377	0.0159150206489777\\
4.3927734395	0.0159150206489777\\
4.4466796895	0.0159150206489777\\
4.5009765645	0.0159150206489777\\
4.502734377	0.184126751512541\\
4.509375002	0.184126751512541\\
4.537109377	0.184126751512541\\
4.5416015645	0.184126751512541\\
4.612109377	0.184126751512541\\
4.634375002	0.184126751512541\\
4.774609377	0.184126751512541\\
4.819140627	0.184126751512541\\
4.8392578145	0.184126751512541\\
4.844531252	0.138885990904913\\
4.879296877	0.138885990904913\\
4.9740234395	0.138885990904913\\
5.024218752	0.138885990904913\\
5.107421877	0.03674945207232\\
5.170703127	0.03674945207232\\
5.205468752	0.03674945207232\\
5.288671877	0.03674945207232\\
5.295703127	0.03674945207232\\
5.3392578145	0.03674945207232\\
5.348046877	0.0183680759440234\\
5.4181640645	0.0183680759440234\\
5.5216796895	0.0183680759440234\\
5.569531252	0.0183680759440234\\
5.5919921895	0.0183680759440234\\
5.5998046895	0.0183680759440234\\
5.6525390645	0.0142509221326956\\
5.727343752	0.0142509221326956\\
5.7447265645	0.0142509221326956\\
5.790234377	0.0142509221326956\\
5.854687502	0.0142509221326956\\
5.886718752	0.00959451552871876\\
};

\end{axis}
\end{tikzpicture}%
}
  \caption{Magnitude of control law in use by agent $i=5$, $|u_5(t)|$ (solid blue), as well as its current promise on the cloud server $M_5(t)$ (dashed red).}\label{fig:promises}
\end{figure}

\section{Conclusion}
We have presented a novel self-triggering algorithm that, given only the ability to communicate asynchronously at discrete intervals through a cloud server, provably drives a set of agents to consensus without Zeno behavior. 
Unlike most previous work, we do not require an agent to be able to listen continuously, instead only being able to receive information at its discrete surfacing times. 
Through the use of control promises, we are able to bound the states of neighboring agents, allowing an agent to remain submerged until its total contribution to the consensus would become detrimental.  
The given algorithm requires no global parameters, and is fully distributed, requiring no computation to be done off of each local platform.
Simulation results show the effectiveness of the proposed algorithm.

In the future, we are interested in investigating control laws different from~\eqref{eq:realcontrol} and forms of $f(x)$ other than $f(x)=cx$ that may be able to provide more infrequent surfacings or faster convergence. 
We are additionally interested in methods to reach approximate consensus rather than true asymptotic consensus, and guaranteeing no Zeno behavior without a dwell time. 

\section*{Acknowledgments}

This work was supported in part by the TerraSwarm Research Center, one of six centers supported by the STARnet phase of the Focus Center Research Program (FCRP) a Semiconductor Research Corporation program sponsored by MARCO and DARPA.

\bibliographystyle{ieeetr}
\bibliography{alias,sean,cameron,cameron-main,Main-add}

\appendix
\section{Proof of Proposition 1}
\label{sec:prop1_proof}

We begin by further splitting up the local objective contribution $\dot{V}_i$ as a sum of individual neighbor pair contributions:
$
\dot{V}_i(t) = \sum_{j\in\NN_i} \dot{V}_{ij}(t) ,
$
where
\begin{align}
\dot{V}_{ij}(t) &\triangleq -u_i(t) \left( x_j(t) - x_i(t) \right) .
\end{align}

For $t \leq \tnext_j$, we can write $\dot{V}_{ij}(t)$ exactly as
\begin{multline}
  \dot{V}_{ij}(t) = -u_i(t) \big[x_j(\tlast_i) - x_i(\tlast_i)  + (u_j(t) - u_i(t)) (t - \tlast_i) \big] .
\label{eq:Vijdot_1}
\end{multline}

For $t > \tnext_j$, since agent~$i$ no longer has access to $u_j(t)$, we write it as follows:
\begin{multline}
  \dot{V}_{ij}(t) = -u_i(t) \Big[ x_j(\tnext_j) + \int_{\tnext_j}^t u_j(\tau)d\tau 
  - (x_i(\tnext_j) + u_i(t)(t-\tnext_j))\Big] .
\end{multline}

We can then use the promise $M_j(t)$ to bound
\begin{align}
\left| \int_{\tnext_j}^t u_j(\tau) d\tau \right| \leq M_j(t) (t-\tnext_j)
\end{align}

allowing us to upper bound $\dot{V}_{ij}(t)$ for $t > \tnext_j$ with
\begin{align}
\dot{V}_{ij}(t) \leq &-u_i(t) (x_j(\tnext_j) - x_i(\tnext_j) ) \nonumber \\
  & + (|u_i(t)| M_j(t) + u_i(t)^2) (t-\tnext_j) .
\label{eq:Vijdot_2}
\end{align}

Letting $\alpha_{ij}$, $\beta_{ij}$, and $\gamma_{ij}$ be as defined in Proposition~\ref{prop:tstar},
 we can write these as 
\begin{align}
\dot{V}_{ij}(t) \leq \left\{ \begin{array}{cc}
                             \alpha_{ij}(\tlast_i) + \beta_{ij}(\tlast_i) (t - \tlast_i) & t \leq \tnext_j \\
                            \alpha_{ij}(\tnext_j) + \gamma_{ij}(\tlast_i) (t-\tnext_j) & \textrm{ otherwise.}
                         \end{array} \right. \label{eq:Vij_bound_1}
\end{align}

Assume that the solution to~\eqref{eq:t_opt} lies in the interval $T_i^* \in [\tnext_{\pi(k)}, \tnext_{\pi(k+1)})$ for some $k \in \{0, \dots, |\mathcal{N}_i| \}$.
On this interval, the states of neighbors $\pi(m)$ for $m>k$ are known exactly, while those $\pi(m')$ with $m'\leq k$ are only known to lie in a ball since they are scheduled to surface and change their control by this time interval.
Using~\eqref{eq:Vijdot_1} and~\eqref{eq:Vijdot_2}, we can write the local objective contribution $\dot{V}_i(t)$ for $t$ in this interval as
\begin{align}
\dot{V}_i(t)
  &= \sum_{m'=1}^k \dot{V}_{i{\pi(m')}}(t) + \sum_{m=k+1}^{|\NN_i|} \dot{V}_{i{\pi(m)}}(t) \\
  &\leq \sum_{m'=1}^k \left( \alpha_{i{\pi(m')}}(\tnext_{\pi(m')}) + \gamma_{i{\pi(m')}}(\tlast_i) (t-\tnext_{\pi(m')})\right) \nonumber \\
      & \ \ \ \ + \sum_{m=k+1}^{|\NN_i|} \left( \alpha_{i{\pi(m)}}(\tlast_i) + \beta_{i{\pi(m)}}(\tlast_i) (t-\tlast_i) \right)
\label{eq:v_ubound}
\end{align}

Let $\tau^{\star,(k)}_i$ be the solution to~\eqref{eq:t_opt} with the additional constraint that $t$ be within the interval $[\tnext_{\pi(k)}, \tnext_{\pi(k+1)}]$:
\begin{equation}
\begin{aligned}
& \underset{t}{\text{infimum}}
& & t \\
& \text{subject to}
& & \max_{x(t) \in R^i(t)} \dot{V}_i(x(t)) > 0 , \\
&&& \tnext_{\pi(k)} \leq t \leq \tnext_{\pi(k+1)} .
\label{eq:t_opt_constrained}
\end{aligned}
\end{equation}

The objective derivative constraint in~\eqref{eq:t_opt_constrained} can be rewritten using~\eqref{eq:v_ubound} in the form seen in Proposition~\ref{prop:tstar}.
The solution to the original optimization~\eqref{eq:t_opt} can then be written as 
\begin{align}
  T^\star_i = \min \ \{\tau^{\star,(k)}_i \ | \ k \in \{0, \ldots, |\NN_i|\} \}
 \label{eq:t_star}
\end{align}

\section{Proof of Proposition 2}
\label{sec:prop2_proof}

First, note that the separation amongst neighbor pairs is still valid: 
\begin{align}
\int_{\tlast_i}^t \dot{V}_i(\tau) d\tau = \sum_{j\in\NN_i} \int_{\tlast_i}^t \dot{V}_{ij}(\tau) d\tau .
\end{align}
For an agent $j \in \NN_i$ and $t \leq \tnext_j$, we can exactly compute the pair contribution over its submerged interval as 
\begin{align}
\int_{\tlast_i}^t \dot{V}_{ij}(\tau)d\tau = \alpha_{ij}(\tlast_i) (t-\tlast_i) + \frac{1}{2} \beta_{ij}(\tlast_i) (t-\tlast_i)^2
\label{eq:vint_expr_1}
\end{align}
where $\alpha_{ij}$ and $\beta_{ij}$ are as given in~\eqref{eq:alpha} and~\eqref{eq:beta}.

For agent $j$ and $t > \tnext_j$, we first split the integral into a part that we can compute exactly and a part that we can only bound:
\begin{align}
\label{eq:vint_split}
\int_{\tlast_i}^t \dot{V}_{ij}(\tau)d\tau = \int_{\tlast_i}^{\tnext_j} \dot{V}_{ij}(\tau)d\tau + \int_{\tnext_j}^t \dot{V}_{ij}(\tau)d\tau ,
\end{align}
and using the bound~\eqref{eq:v_ubound}, can write
\begin{align}
 \int_{\tnext_j}^t \dot{V}_{ij}(\tau)d\tau \leq \alpha_{ij}(\tnext_j)(t-\tnext_j) + \frac{1}{2} \gamma_{ij}(\tlast_i) (t-\tnext_j)^2 .
\end{align}

	The total pair contribution $\int \dot{V}_{ij}(t)$ is then given by~\eqref{eq:vint_expr_1}
for $t \leq \tnext_j$, and bounded by
\begin{multline}
\int_{\tlast_i}^t \dot{V}_{ij}(\tau)d\tau \leq
                \alpha_{ij}(\tlast_i) (\tnext_j-\tlast_i) 
  + \frac{1}{2} \beta_{ij}(\tlast_i) (\tnext_j-\tlast_i)^2 + \\ \alpha_{ij}(\tnext_j)(t-\tnext_j)
  + \frac{1}{2} \gamma_{ij}(\tlast_i) (t-\tnext_j)^2
\label{eq:v_int_ubound2}
\end{multline}
for $t > \tnext_j$.

As before, consider times $t$ that lie in the interval $[\tnext_{\pi(k)}, \tnext_{\pi(k+1)})$ for some $k \in \{0, \dots, |\NN_i|\}$.
We can bound the full objective contribution for times $t$ in this interval as (from~\eqref{eq:vint_expr_1},~\eqref{eq:vint_split}, and~\eqref{eq:v_int_ubound2})
\begin{multline}
\label{eq:vint_bound_final} 
\int_{\tlast_i}^t \dot{V}_i(\tau) d\tau\leq  \sum_{m' = 1}^k \Big[ \alpha_{i\pi(m')} (\tlast_i) (\tnext_{\pi(m')}-\tlast_i) 
+ \frac{1}{2} \beta_{i{\pi(m')}}(\tlast_i) (\tnext_{\pi(m')}-\tlast_i)^2 \\
                     \ + \alpha_{i{\pi(m')}}(\tnext_{\pi(m')})(t-\tnext_{\pi(m')}) 
                     + \frac{1}{2} \gamma_{i{\pi(m')}}(\tlast_i) (t-\tnext_{\pi(m')})^2 \Big] \\
    + \sum_{m=k+1}^{|\NN_i|}  \Big[ \alpha_{i{\pi(m)}} (\tlast_i) (t-\tlast_i)
    + \frac{1}{2} \beta_{i{\pi(m)}}(\tlast_i) (t-\tlast_i)^2 \Big] .
\end{multline}

Let $\tau^{tot,(k)}_i$ be the optimal solution to~\eqref{eq:t_int_opt} with the additional constraint that $t$ be within the interval $[\tnext_{\pi(k)}, \tnext_{\pi(k+1)})$.
Using~\eqref{eq:vint_bound_final}, we can rewrite the objective bound in~\eqref{eq:t_int_opt} on this interval, resulting in the optimization seen in~\eqref{eq:vint_opt_constrained_prop}.
The solution $T^{\textrm{total}}_i$ to~\eqref{eq:t_int_opt} can then be written as 
\begin{align}
T^{\textrm{total}}_i = \min \ \{\tau^{tot,(k)}_i \ | \ k \in \{0, \dots, |\NN_i\} \}
\end{align}

\end{document}